\documentclass[12pt,a4paper]{article}
\usepackage{amsmath,amssymb,amsthm,graphicx,color,accents}
\usepackage[left=1in,right=1in,top=1in,bottom=1in]{geometry}
\usepackage{cite}
\usepackage[british]{babel}
\usepackage{enumerate}
\usepackage{hyperref} 
\usepackage{bm}
\usepackage{comment}

\hypersetup{
    colorlinks=false,
    pdfborder={0 0 0},
}

\newtheorem{theorem}{Theorem}[section]
\newtheorem{corollary}[theorem]{Corollary}
\newtheorem{proposition}[theorem]{Proposition}
\newtheorem{lemma}[theorem]{Lemma}
\newtheorem{conjecture}[theorem]{Conjecture}
 
\theoremstyle{definition}

\theoremstyle{remark}
\newtheorem{remark}[theorem]{Remark}
\newtheorem{remarks}[theorem]{Remarks}

 \allowdisplaybreaks

\newcommand{\1}[1]{{\mathbf 1}{\{#1\}}}
\newcommand{\2}[1]{{\mathbf 1}{(#1)}}

\newcommand{\R}{{\mathbb R}}
\newcommand{\Z}{{\mathbb Z}}
\newcommand{\N}{{\mathbb N}}

\newcommand{\RP}{{\mathbb R}_+}
\newcommand{\Sp}{{\mathbb S}}

\DeclareMathOperator{\Exp}{\mathbb{E}}
\renewcommand{\Pr}{{\mathbb P}}
\DeclareMathOperator{\Var}{\mathbb{V}ar}

\DeclareMathOperator*{\argmin}{arg \mkern 1mu \min}
\DeclareMathOperator*{\argmax}{arg \mkern 1mu \max}
\DeclareMathOperator*{\diam}{diam}

\DeclareMathOperator{\trace}{tr}

\newcommand{\hull}{\mathop \mathrm{hull}}
\newcommand{\sd}{\mathop\triangle}

\newcommand{\per}{{\mkern -1mu \scalebox{0.5}{$\perp$}}}
\newcommand{\sperp}{\sigma^2_{\mu_\per}}
\newcommand{\spara}{\sigma^2_{\mu}}

\newcommand{\tra}{{\scalebox{0.6}{$\top$}}}

\newcommand{\eps}{\varepsilon}

\newcommand{\rc}{{\mathrm{c}}}
\newcommand{\ud}{{\mathrm d}}
\newcommand{\cA}{{\mathcal A}}
\newcommand{\cB}{{\mathcal B}}

\newcommand{\cD}{{\mathcal D}}

\newcommand{\cF}{{\mathcal F}}

\newcommand{\cH}{{\mathcal H}}
\newcommand{\cK}{{\mathcal K}}
\newcommand{\cL}{{\mathcal L}}
\newcommand{\cN}{{\mathcal N}}
\newcommand{\cR}{{\mathcal R}}

\newcommand{\cT}{{\mathcal T}}

\newcommand{\as}{\ \text{a.s.}}
\newcommand{\io}{\ \text{i.o.}}

\newcommand{\tod}{\overset{\textup{d}}{\longrightarrow}}

\newcommand{\bx}{{\mathbf{x}}}
\newcommand{\by}{{\mathbf{y}}}

\newcommand{\be}{{\mathbf{e}}}
\newcommand{\0}{{\mathbf{0}}}

\newcommand{\ubar}[1]{\underaccent{\bar}{#1}}
\newcommand{\unittheta}{\mathbf{e}_{\theta}}
\newcommand{\unitzero}{\mathbf{e}_{0}}
\newcommand{\unithalfpi}{\mathbf{e}_{\pi/2}}

\makeatletter
\def\namedlabel#1#2{\begingroup  
    (#2)%
    \def\@currentlabel{#2}%
    \phantomsection\label{#1}\endgroup
}
\makeatother

\begin{document}

\title{The convex hull of a planar random walk:\\
 perimeter, diameter, and shape}
\author{James McRedmond \and Andrew R.\ Wade}

\date{\today}
\maketitle

\begin{abstract}
We study the convex hull of the first $n$ steps of a planar random walk,
and present large-$n$ asymptotic results on its perimeter length $L_n$,   diameter $D_n$,
and   shape. In the case where the walk has a non-zero mean drift, we show that $L_n / D_n \to 2$ a.s., and give
distributional limit theorems and variance asymptotics for $D_n$, and in the zero-drift case we show that the convex hull
is infinitely often arbitrarily well-approximated in shape by any unit-diameter compact convex set containing the origin,
and then $\liminf_{n \to \infty} L_n/D_n =2$ and $\limsup_{n \to \infty} L_n /D_n = \pi$, a.s. Among the tools that
we use is a zero-one law for convex hulls of random walks.
  \end{abstract}

\medskip


\medskip


\section{Model and main results}

\subsection{Introduction and notation}

The geometry of random walks in Euclidean space
is a topic of persistent interest (see e.g.~\cite{rg}). The \emph{convex hull}
of a random walk is a classical geometrical characteristic of the walk~\cite{sw,ss} that has received renewed attention recently~\cite{kvz1,kvz2,ty,wx1,wx2,vz,x};
see~\cite{mcr} for a recent survey, including motivation in terms of modelling  the home range of roaming animals.
The present paper studies the asymptotic behaviour of the convex hull of
a random walk in $\R^2$, focusing on its \emph{shape}, its \emph{perimeter length}, and its \emph{diameter}.

Let $Z, Z_1, Z_2, \ldots$ be i.i.d.~random variables in $\R^2$,
and let $S_0, S_1, S_2, \ldots$ be the associated random walk, defined by $S_0 := \0$ (the origin in $\R^2$) and 
 $S_n := \sum_{k=1}^n Z_k$ for $n \in\N$.
Let $\cH_n := \hull \{ S_0, S_1, \ldots, S_n \}$, where $\hull A$ denotes the convex hull of $A \subseteq \R^2$.
Write $L_n$ for the perimeter length of $\cH_n$,
and let
\[ D_n := \diam \{ S_0, S_1, \ldots, S_n\} = \diam \cH_n. \]
(Note that $\cH_n$, $L_n$, $D_n$ are all random variables on the appropriate spaces: see the comments
at the start of Section~\ref{sec:zero-one}.)

A striking early result on $L_n$ is the formula of Spitzer and Widom~\cite{sw}:
\begin{equation}
\label{s-w} 
\Exp L_n = 2 \sum_{k=1}^n \frac{1}{k} \Exp \| S_k \| ,
\end{equation}
where,  and subsequently, $\| \, \cdot \, \|$ is the Euclidean norm on $\R^2$.
For $\bx \in \R^2 \setminus \{ \0\}$ we set $\hat \bx := \bx / \| \bx \|$.
If $\Exp \| Z \| < \infty$, we write $\mu := \Exp Z$.
If $\Exp ( \| Z \|^2 ) < \infty$, we write
 $\sigma^2  := \Exp (   \| Z - \mu \|^2 )$, and if $\mu \neq \0$ we write
\[ \spara := \Exp [ ( ( Z - \mu ) \cdot \hat \mu  )^2 ], \text{ and } \sperp := \sigma^2 - \spara. \]

The results in this paper concern the asymptotics of $L_n$, $D_n$, and the shape of $\cH_n$.
The cases where $\mu =\0$ and $\mu \neq \0$ are, as is no surprise, quite different.
Rough information about the shape of $\cH_n$ is given by the ratio $L_n/D_n$;
provided that $\Pr ( Z = \0 ) < 1$, convexity implies that a.s., for all but finitely many $n$,
\begin{equation}
\label{l-d-ineq}
2 \leq {L_n}/{D_n} \leq \pi , 
\end{equation}
the extrema being the line segment and shapes of constant width (such as the disc).

\subsection{Laws of large numbers}

Our first result is the following law of large numbers for $L_n$.

\begin{theorem}
\label{t:ss_lln}
Suppose that $\Exp \| Z  \| < \infty$. Then
\[ \lim_{n \to \infty} n^{-1} L_n = 2 \| \mu \| , \text{ a.s.~and in } L^1 .\]
On the other hand, if $\Exp \| Z  \| = \infty$ then $\limsup_{n \to \infty} n^{-1} L_n = \infty$, a.s.
\end{theorem}
\begin{remark}
Snyder and Steele~\cite{ss} obtained the almost-sure convergence result under the stronger condition $\Exp ( \| Z \|^2 ) < \infty$
as a consequence of an upper bound on $\Var L_n$ deduced from Steele's version of the Efron--Stein inequality. (Snyder and Steele state their result for the case $\mu  \neq \0$, but their proof works equally well
when $\mu = \0$.) 
\end{remark}

Similarly, we have a law of large numbers for $D_n$.

\begin{theorem}
\label{t:diam_lln}
Suppose that $\Exp \| Z  \| < \infty$. Then
\[ \lim_{n \to \infty} n^{-1} D_n = \| \mu \| , \text{ a.s.~and in } L^1 .\]
On the other hand, if $\Exp \| Z  \| = \infty$ then $\limsup_{n \to \infty} n^{-1} D_n = \infty$, a.s.
\end{theorem}

In the case $\mu \neq \0$, 
Theorems~\ref{t:ss_lln} and~\ref{t:diam_lln} have the following immediate consequence.

\begin{corollary}
\label{cor:ratio-drift}
Suppose that $\Exp \| Z  \| < \infty$ and that $\mu \neq \0$.
Then 
\[ \lim_{n \to \infty} L_n / D_n = 2, \as \] 
\end{corollary}

\subsection{Zero-drift case}

In the zero-drift case, we need the extra condition $\Exp ( \| Z \|^2 ) < \infty$.
Let $\Sigma := \Exp ( Z Z^\tra )$, viewing $Z$ as a column vector; note that if $\mu = \0$
then $\trace \Sigma = \sigma^2$.  Our first result is the following shape theorem.
Let $\rho_H$ denote the Hausdorff distance between non-empty compact sets;
see~\eqref{eq:rhoH} below for a definition.

\begin{theorem}\label{thm:shape}
Suppose that $\Exp ( \| Z \|^2 ) < \infty$, $\Sigma$ is positive definite, and $\mu   = \0$. Then,
for any compact convex set $K \subset \R^2$ with $\diam K = 1$,
\[ \liminf_{n\to \infty} \rho_H(D_n^{-1}\cH_n,K) =0, \as \]
\end{theorem}

Note that under the hypotheses of Theorem~\ref{thm:shape}, $\Pr ( Z = \0 ) < 1$, so that $D_n > 0$ for all but finitely many $n$, a.s.
A consequence of Theorem~\ref{thm:shape} is the following result, which should be contrasted
with Corollary~\ref{cor:ratio-drift}.

\begin{corollary}\label{cor:infsupLD}
Suppose that $\Exp ( \| Z \|^2 ) < \infty$, $\Sigma$ is positive definite, and $\mu   = \0$. Then,
\[ 2 = \liminf_{n \to \infty} \frac{L_n}{D_n} < \limsup_{n \to \infty} \frac{L_n}{D_n} = \pi, \as \]
\end{corollary}

\subsection{Case with drift}

Now we turn  to the individual asymptotics for $L_n$ and $D_n$ in the case with non-zero drift. 
The  behaviour of $L_n$ was studied in~\cite{wx1}, where it was shown (Theorem~1.3 of~\cite{wx1}) that
if  $\Exp ( \| Z \|^2 ) < \infty$ and $\mu \neq \0$, then, as $n \to \infty$,
\begin{equation}
\label{wx}
 n^{-1/2} | L_n - \Exp L_n - 2 ( S_n - \Exp S_n ) \cdot \hat \mu | \to 0 , \text{ in } L^2 .
\end{equation}
As shown in~\cite{wx1}, this result implies variance asymptotics for $L_n$
as well as a central limit theorem when $\spara >0$.
We show that~\eqref{wx} may be recast in the following stronger form.

\begin{theorem}
\label{thm1}
Suppose that $\Exp ( \| Z  \|^2 ) < \infty$ and $\mu \neq \0$. Then, as $n \to \infty$,
\[ n^{-1/2} | L_n - 2 S_n \cdot \hat \mu | \to 0, \text{ in } L^2. \]
\end{theorem}

The following result is the key additional component in the proof of Theorem~\ref{thm1},
and is of interest in its own right; its proof uses the Spitzer--Widom formula~\eqref{s-w}.

\begin{theorem}
\label{thm:drift-mean}
Suppose that $\Exp ( \| Z  \|^2 ) < \infty$ and $\mu \neq \0$.
Then, as $n \to \infty$,
\[  \Exp   L_n =  2  \| \mu \| n + \left(\frac{\sperp}{  \| \mu \|} + o(1) \right) \log n . \]
\end{theorem}

Analogous asymptotic expansions for $\Exp L_n$ in the case $\mu = \0$ have been
presented recently in~\cite{glm}.
For the diameter $D_n$, we have the following analogue of Theorem~\ref{thm1}.

\begin{theorem}
\label{thm2}
Suppose that $\Exp ( \| Z  \|^2 ) < \infty$ and $\mu \neq \0$. Then, as $n \to \infty$, 
\[ n^{-1/2} | D_n -   S_n \cdot \hat \mu | \to 0, \text{ in } L^2. \]
\end{theorem}

Denote by $\cN(0,1)$ the standard normal distribution, and by $\tod$ convergence in distribution.
Theorem~\ref{thm2}  yields variance asymptotics and a central limit theorem when $\spara >0$, as follows.

\begin{corollary}
\label{cor:diam-clt}
Suppose that $\Exp ( \| Z  \|^2 ) < \infty$ and $\mu \neq \0$. Then $\lim_{n \to \infty} n^{-1} \Var D_n = \spara$.
Moreover, if $\spara > 0$, for $\zeta \sim \cN (0,1)$, as $n \to \infty$,
\[ \frac{ D_n - \Exp D_n}{\sqrt{ \Var D_n }} \tod \zeta, \text{ and } \frac{ D_n - n \| \mu \|}{\sqrt{ n \spara }} \tod \zeta .\]
\end{corollary}

The degenerate case $\spara =0$ corresponds to the case where $Z \cdot \hat \mu = \| \mu \|$ a.s., and is of its own interest. It includes, for example, the case where $Z = (1,1)$ or $(1,-1)$, each with probability $1/2$, in which the two-dimensional walk $S_n$ corresponds to the space-time diagram of a one-dimensional simple symmetric random walk.
In the case $\spara = 0$, Corollary~\ref{cor:diam-clt} says only that $\Var D_n = o(n)$. We prove the following sharper result
which requires some additional conditions.

\begin{theorem}\label{thm3}
Suppose that $\Exp ( \| Z  \|^p ) < \infty$ for some $p>2$, 
 $\mu \neq \0$, and $\spara = 0$. Then,
\begin{equation}\label{eqn4}
D_n - \| \mu \| n \tod \frac{\sperp \zeta^2}{2 \| \mu \|} ,
\end{equation}
where $\zeta \sim \cN (0,1)$. Further, if, in addition, $\Exp(\| Z \|^p) < \infty$ for some $p>4$, then
\begin{equation} \label{eqn5}
\lim_{n\to \infty} \Var D_n = \frac{{\sigma^4_{\mu_\per}}}{2 \| \mu \|^2}.
\end{equation}
\end{theorem}
\begin{remarks} \begin{itemize}
\item[(i)]
The higher moments conditions required in Theorem~\ref{thm3}
are necessary for the proofs that we employ; see also~Remark~\ref{p2condition} below.
\item[(ii)] The statement~\eqref{eqn4}   may be written as 
	\begin{equation}
	\label{eq8}
	D_n -  S_n \cdot \hat{\mu}  \tod \frac{\sigma_{\mu_{\perp}}^2 \zeta^2}{2\|\mu\|}.\end{equation}
		It is natural to ask whether~\eqref{eq8} also holds in the case where $\sigma_\mu^2 >0$;
	if it did, then it would provide an alternative proof of the central limit theorem in Corollary~\ref{cor:diam-clt}.
Simulations suggest that when $\sigma_\mu^2 >0$, equation~\eqref{eq8} holds in some, but not all cases.
	\end{itemize}
\end{remarks}

\subsection{Open problems and paper outline}

When $\Exp ( \| Z  \|^2 ) < \infty$, $\mu \neq \0$, and $\spara =0$, Theorem~\ref{thm1} 
(see also Theorem~1 in~\cite{wx1})
shows that $\Var L_n = o(n)$. It was conjectured in~\cite{wx1} that $\Var L_n = O (\log n)$ in this case,
which is the subject of ongoing work.
We make the following stronger conjecture. 

\begin{conjecture}
Suppose that $\Exp ( \| Z  \|^2 ) < \infty$, $\mu \neq \0$, $\spara = 0$, and $\sperp >0$. Then 
\[ \lim_{n \to \infty} \frac{\Var L_n}{\log n} \text{ exists in } (0,\infty). \]
\end{conjecture}
 
The outline of the remainder of the paper is as follows. 
In Section~\ref{sec:lln} we give the proofs of the laws of large numbers
Theorems~\ref{t:ss_lln} and~\ref{t:diam_lln}. In Section~\ref{sec:zero-one}
we present a zero-one law for the convex hull of random walk (Theorem~\ref{thm:zero-one}),
which we then use to prove Theorem~\ref{thm:shape} and Corollary~\ref{cor:infsupLD}.
Section~\ref{sec:perim-drift} presents the proofs of Theorems~\ref{thm1} and~\ref{thm:drift-mean}.
Sections~\ref{sec:diam-drift} and~\ref{sec:diam-degen} give the proofs of Theorems~\ref{thm2} and~\ref{thm3} respectively.
Finally, 
rather than interrupting the flow of the main arguments, we present in Appendix~\ref{sec:appendix}
a couple of auxiliary technical results.

\section{Laws of large numbers}
\label{sec:lln}

Throughout we use the notation $\be_\theta = (\cos \theta, \sin \theta)$ for the unit vector
in direction $\theta$. We recall (see e.g.~equation~(2.1) of~\cite{ss})
that \emph{Cauchy's formula} states that 
for a finite point set $\{ \bx_0, \bx_1, \ldots, \bx_n \} \subset \R^2$,
the perimeter length of $\hull \{ \bx_0, \bx_1, \ldots, \bx_n \}$
is given by
\[ \int_0^{2\pi} \max_{0 \leq k \leq n} ( \bx_k \cdot \be_\theta ) \ud \theta .\]

\begin{proof}[Proof of Theorem~\ref{t:ss_lln}.]
Cauchy's formula applied to our random walk implies that
\begin{equation}
\label{eq:cauchy}
 L_n = \int_{0}^{2\pi} \max_{0 \leq k \leq n } ( S_k \cdot \be_\theta ) \ud \theta .\end{equation}

First suppose that $\Exp \| Z  \| < \infty$. 
Then the strong law of large numbers says that for any $\eps>0$
there exists $N_\eps$ with $\Pr ( N_\eps < \infty ) =1$ for which
\begin{equation}
\label{eq:lln}
\| S_n - n \mu \| < n \eps , \text{ for all } n \geq N_\eps .\end{equation}
Since $S_0 =\0$, taking $k=0$ and $k=n$ in~\eqref{eq:cauchy} and writing $x^+ := x \1 { x> 0}$, we have 
\begin{align}
\label{eq:lower-bound}
L_n  \geq \int_{0}^{2\pi} (S_n \cdot \be_\theta )^+ \ud \theta = 2 \| S_n \| ,
\end{align}
by Cauchy's formula for $\hull \{ \0, S_n \}$. For $n \geq N_\eps$
we have from~\eqref{eq:lln} that 
\[ \| S_n \| \geq \| n \mu \| - \| S_n - n \mu \| \geq n \| \mu \| - n \eps .\]
Since $\eps >0$ was arbitrary, it follows that
$\liminf_{n \to \infty} n^{-1} L_n \geq 2 \| \mu \|$, a.s.

On the other hand, for any $\eps >0$, we have from~\eqref{eq:lln} that 
\begin{align*}
\max_{0 \leq k \leq n} (S_k \cdot \be_\theta ) & \leq \max_{0 \leq k \leq N_\eps} (S_k \cdot \be_\theta ) +  \max_{N_\eps \leq k \leq n} (S_k \cdot \be_\theta ) \\
& \leq \max_{0 \leq k \leq N_\eps} \| S_k \| + \max_{0 \leq k \leq n} \left( k (\mu \cdot \be_\theta + \eps) \right) \\
& = \max_{0 \leq k \leq N_\eps} \| S_k \| +   n  ( \mu \cdot \be_\theta +   \eps)^+ .
\end{align*}
Let $A_\eps := \{ \theta \in [0,2\pi] : \mu \cdot \be_\theta > - \eps \}$.
Then
\[ \int_0^{2\pi} ( \mu \cdot \be_\theta + \eps )^+ \ud \theta = \int_{A_\eps} ( \mu \cdot \be_\theta + \eps ) \ud \theta
\leq \int_{A_\eps}  \mu \cdot \be_\theta \ud \theta + 2\pi \eps .\]
But
\begin{align*} \int_{A_\eps}  \mu \cdot \be_\theta \ud \theta
&  = \int_{A_0} \mu \cdot \be_\theta \ud \theta + \int_{A_\eps \setminus A_0}  
\mu \cdot \be_\theta \ud \theta \\
& \leq \int_0^{2\pi} ( \mu \cdot \be_\theta )^+ \ud \theta +  \| \mu \| | A_\eps \setminus A_0 |
.\end{align*}
Hence, from~\eqref{eq:cauchy} we obtain
\[ L_n \leq 2 \pi \max_{0 \leq k \leq N_\eps} \| S_k \|  + n \int_0^{2\pi} ( \mu \cdot \be_\theta )^+ \ud \theta + 2 \pi n \eps +  n \| \mu \| | A_\eps \setminus A_0 |.\]
Since $\Pr ( N_\eps < \infty ) = 1$, 
it follows from Cauchy's formula for $\hull \{ \0, \mu \}$ that, a.s.,
\[ \limsup_{n \to \infty} n^{-1} L_n \leq 2 \| \mu \| + 2 \pi  \eps +  \| \mu \| | A_\eps \setminus A_0 | .\]
Since $\eps >0$ was arbitrary, and $| A_\eps \setminus A_0 | \to 0$ as
$\eps \to 0$, we get $\limsup_{n \to \infty} n^{-1} L_n \leq 2 \| \mu \|$, a.s.
Thus the almost sure convergence statement is established. 

Moreover, from~\eqref{eq:cauchy},
\begin{align*} L_n & \leq \int_0^{2\pi} \max_{0 \leq k \leq n} \| S_k \| \ud \theta \\
& \leq 2 \pi  \max_{0 \leq k \leq n} \sum_{j=1}^k \| Z_j \| \\
& \leq 2 \pi \sum_{j=1}^n \| Z_j \| .\end{align*}
The strong law shows that, a.s.,
$n^{-1} \sum_{j=1}^n \| Z_j \| \to \Exp \| Z  \| < \infty$,
while $\Exp ( n^{-1} \sum_{j=1}^n \| Z_j \|) = \Exp \| Z \|$; hence
Pratt's lemma \cite[p.~221]{gut} implies that $n^{-1} L_n \to 2 \| \mu \|$ in $L^1$.

Finally, suppose that $\Exp \| Z  \| = \infty$. 
From~\eqref{eq:lower-bound}, it suffices to show that 
\[ \limsup_{n \to\infty} n^{-1} \| S_n \| =\infty, \text{ a.s.} \]
To this end we follow~\cite[p.~297]{gut}.
First (see e.g.~\cite[p.~75]{gut}) $\Exp \| Z  \| = \infty$ implies that
for any $c >0$, we have
$\sum_{n =1}^\infty \Pr ( \| Z_n \| \geq c n ) = \infty$, which,
by the Borel--Cantelli lemma, implies that $\Pr ( \| Z_n \| \geq c n \text{ i.o.} ) =1$.
But $\| Z_n \| \leq \| S_n \| + \| S_{n-1} \|$, so it follows that
$\Pr ( \| S_n \| \geq cn/2 \text{ i.o.} ) = 1$. In other words,
$\limsup_{n \to \infty} n^{-1} \| S_n \| \geq c/2$, a.s., and,
since $c>0$ was arbitrary, we get the result.
\end{proof}

\begin{proof}[Proof of Theorem~\ref{t:diam_lln}.]
Since $\| S_n \| \leq D_n \leq L_n /2$ we can apply the strong law for $S_n$,
which implies that $n^{-1} \| S_n \| \to \| \mu \|$, and Theorem~\ref{t:ss_lln},
to deduce that $n^{-1} D_n \to \| \mu \|$, a.s. Since $n^{-1} D_n \leq n^{-1} L_n /2$
we may again apply Pratt's lemma \cite[p.~221]{gut} to deduce the $L^1$ convergence.
Finally, if $\Exp \| Z  \| = \infty$ we use the bound $D_n \geq L_n / \pi$ and the final
statement in Theorem~\ref{t:ss_lln} to deduce that $\limsup_{n \to \infty} n^{-1} D_n = \infty$, a.s.
\end{proof}

\section{A zero-one law for convex hulls}
\label{sec:zero-one}

A key ingredient in the proof of Theorem~\ref{thm:shape} is a \emph{zero-one law} (Theorem~\ref{thm:zero-one} below).
Before we state the result, we need some notation.
Define $\sigma$-algebras $\cF_0 := \{ \emptyset, \Omega \}$ and $\cF_n := \sigma (Z_1, \ldots, Z_n)$ for $n \geq 1$; also set
$\cF_\infty := \sigma ( \cup_{n \geq 0} \cF_n )$.
 Let $\rho_d$ denote the Euclidean metric on $\R^d$,
and for $A \subseteq \R^d$ and $\bx \in \R^d$, let $\rho_d ( \bx, A) := \inf_{\by \in A} \rho_d ( \bx, \by)$. 

Let $\cK$ denote the set of compact convex subsets of $\R^2$ containing the origin, endowed
with the Hausdorff metric $\rho_H$ defined for $K_1, K_2 \in \cK$ by
\begin{equation}
\label{eq:rhoH}
 \rho_H ( K_1, K_2) = \inf \{ \eps \geq 0 : K_1 \subseteq K_2^\eps \text{ and } K_2 \subseteq K_1^\eps \} , \end{equation}
where $K^\eps := \{ \bx \in \R^2 : \rho_2 ( \bx, K) \leq \eps \}$.
The metric $\rho_H$ generates the associated Borel $\sigma$-algebra
$\cB ( \cK )$.
Since the function $(\bx_0, \bx_1, \ldots, \bx_n ) \mapsto
\hull \{ \bx_0, \bx_1, \ldots, \bx_n \}$ (with $\bx_0:=\0$) is continuous from $(\R^{2(n+1)},\rho_{2(n+1)})$ to $(\cK,\rho_H)$, it is measurable
from $(\R^{2(n+1)}, \cB (\R^{2(n+1)}) )$ to $(\cK, \cB (\cK))$; thus $\cH_n$ is a $\cK$-valued random variable,
and $\cH_n$ is $\cF_n$-measurable.

For $n \geq 0$, set $\cT_n := \sigma ( \cH_n , \cH_{n+1}, \ldots )$
and define $\cT := \cap_{n \geq 0} \cT_n$.
Also, for $n \geq 0$ define
 \[ r_n := \inf \{ \| \bx \| : \bx \in \R^2 \setminus \cH_n \}. \]
Note that $r_n$ is non-decreasing. Here is the zero-one law.

\begin{theorem}
\label{thm:zero-one}
Suppose that $r_n \to \infty$ a.s.
Then if $A \in \cT$, $\Pr (A) \in \{0,1\}$.
\end{theorem}

Next we give a sufficient condition for $r_n \to \infty$.
Recall~\cite[p.~190]{dur} that $S_n$ is \emph{recurrent}
if there is a non-empty set $\cR$ of points $\bx \in \R^2$ (the recurrent values) such that, for any $\eps >0$,
$\| S_n - \bx \| < \eps$ i.o., a.s.

\begin{proposition}
\label{prop:recurrence}
If $S_n$ is genuinely 2-dimensional and recurrent, then $r_n \to \infty$ a.s.
\end{proposition}

\begin{remark}
One may also have $r_n \to \infty$ a.s.~in the case of a transient walk, provided
it visits all angles. However, 
$\lim_{n \to \infty} r_n < \infty$ a.s.~may occur
 if the walk has a limiting direction, such as if there is a finite non-zero drift.
\end{remark}

Let  $B(\bx;r)$ denote the closed Euclidean ball centred at $\bx\in \R^2$ with radius $r$.

\begin{proof}[Proof of Proposition~\ref{prop:recurrence}.]
Since $S_n$ is recurrent, the set $\cR$ of recurrent values is a closed subgroup of $\R^2$
and coincides with the set of \emph{possible values} for the walk: see~\cite[p.~190]{dur}.
Since $S_n$ is genuinely 2-dimensional, it follows from e.g.~Theorem 21.2 of~\cite[p.~225]{br} that $\cR$ 
contains a further closed subgroup $\cR'$ of the form $H \Z^2$ where $H$ is a 
non-singular 2 by 2 matrix. 
Hence there exists $h >0$ such that for every $\bx \in \R^2$ there exists $\by \in \cR'$ with $\| \bx - \by \| < h$.
In particular, for any $\bx \in \R^2$, $\Pr (S_n \in B(\bx;h) \io )=1$.

Fix $r > h$, and consider $4$ discs, $D_1,D_2,D_3,D_4$, each of radius $h$, centred at 
$(\pm 2r, \pm 2r )$. Define $T_r$ to be the first time at which the walk has visited all $4$ discs, i.e.,
\[ T_r := \min\{ n \geq 0 : \exists\ i_1,i_2,i_3,i_4 \in [0, n] \text{ with } S_{i_j}\in D_j \text{ for } j =1,2,3,4 \}. \]
The first paragraph of this proof shows that $T_r < \infty$ a.s. By construction, for $n \geq T_r$
we have that $\cH_n$ contains the square $[-r,r]^2$, and so
$n \geq T_r$ implies $r_n \geq r$.
 Hence, 
\[ \Pr \left( \liminf_{m \to \infty} r_m \geq r \right) \geq \Pr ( T_r \leq n) \to 1 ,\]
as $n \to \infty$, and so $\liminf_{n \to \infty} r_n \geq r$, a.s. Since $r > h$
was arbitrary, the result follows.
\end{proof}

The first step in the proof of Theorem~\ref{thm:zero-one} is the following result,
which uses the fact that $r_n \to \infty$ to show that any initial segment
of the trajectory is eventually contained in the interior of the convex hull, uniformly
over permutations of the initial increments.

\begin{lemma}
\label{lem:recurrence}
Suppose that $r_n \to \infty$ a.s.
Let $k \in \N$. Then there exists a random variable
$N_k$ with $\Pr (k < N_k < \infty) =1$ such that (i) $N_k$ is invariant under permutations of
$Z_1, \ldots, Z_k$, and (ii) $\cH_n = \hull \{ S_{k+1}, \ldots, S_n \}$ for all $n \geq N_k$.
\end{lemma}
\begin{proof}
Fix $k \in \N$.
Let $R_k := \sum_{i=1}^k \|Z_i\|$ and define $N_k := \min \{ n > k : r_n > R_k \}$. Note
that since $r_n$ is non-decreasing, $n \geq N_k$ implies $r_n > R_k$.
Since $R_k < \infty$ a.s.~and $r_n \to \infty$ a.s., we have $N_k < \infty$ a.s.
Observe that if $r_n > R_k$ for $n > k$, then
$S_0 , S_1, \ldots, S_k$ are all contained in the interior
of $\cH_n$, so that $\cH_n = \cH_{n,k} := \hull \{ S_{k+1}, \ldots , S_n \}$.
So statement~(ii) holds. Moreover, if $r_{n,k} := \inf \{ \| \bx \| : \bx \in \R^2 \setminus \cH_{n,k} \}$
we have that $\{ r_n > R_k \} = \{ r_{n,k} > R_k \}$.
But the  events $\{ r_{n,k} > R_k \}$, $n > k$, which determine $N_k$,
depend  only on $R_k$ and $S_{k+1}, S_{k+2}, \ldots$, and so statement~(i) holds.
\end{proof}

Heuristically, Theorem~\ref{thm:zero-one} is true since any $A \in \cT$ is determined
by $\cH_{N_k}, \cH_{N_k+1}, \ldots$, and Lemma~\ref{lem:recurrence} shows that this sequence in invariant under permutations of $Z_1, \ldots, Z_k$,
as required for the Hewitt--Savage zero-one law. The formal proof is as follows.

\begin{proof}[Proof of Theorem~\ref{thm:zero-one}.]
We adapt one of the standard proofs of the Hewitt--Savage zero-one law; see e.g.~\cite[pp.~180--181]{dur}. Let $A \in \cT$ and fix $\eps >0$. Recall a fact from measure theory:
if $\cA$ is an algebra and $A \in \sigma (\cA)$, then we can find $A' \in \cA$
such that $\Pr ( A \sd A' ) < \eps$ (see e.g.~\cite[p.~179]{bill}). Applied to the algebra $\cup_{n \geq 0} \cF_n$ which generates $\cF_\infty \supseteq \cT$, this result
implies that we can find $k \geq 0$ and $A_k \in \cF_k$ such that $\Pr ( A \sd A_k ) < \eps$. 
Fix this $k$, and fix $n$ such that $\Pr ( N_{2k} > n ) < \eps$, where $N_{2k}$ is as given in Lemma~\ref{lem:recurrence}.
Applied to the algebra
$\cA_n := \cup_{m \geq 0} \sigma ( \cH_n, \cH_{n+1}, \ldots, \cH_{n+m} )$, which has $\sigma (\cA_n) \supseteq \cT_n \supseteq \cT$,
the same measure-theoretic result shows that we can find
$E_n \in \cA_n$ such that $\Pr ( A \sd E_n ) < \eps$.

Now $A_k \in \cF_k$ can be expressed as $A_k = \{ Z_1 \in C_{k,1}, \ldots, Z_k \in C_{k,k} \}$
for Borel sets $C_{k,1}, \ldots, C_{k,k}$.
Set $A'_k := \{ Z_{k+1} \in C_{k,1}, \ldots, Z_{2k} \in C_{k,k} \}$;
since the $Z_i$ are i.i.d., $\Pr (A'_k) = \Pr (A_k)$, and $A_k$ and $A_k'$ are independent.
We claim that 
\begin{equation}
\label{eqn:zeroone2}
\Pr ( ( A'_k \sd E_n ) \cap \{ N_{2k} \leq n \} ) = \Pr ( ( A_k \sd E_n ) \cap \{ N_{2k} \leq n \} ) \leq 2\eps .\end{equation}
To see the equality in~\eqref{eqn:zeroone2}, observe that Lemma~\ref{lem:recurrence} shows that $E_n \cap \{ N_{2k} \leq n \}$ is invariant under permutations of $Z_1, \ldots, Z_{2k}$,
and the $Z_i$ are i.i.d. For the inequality in~\eqref{eqn:zeroone2}, we use the 
fact that $\Pr ( A \sd B ) \leq \Pr ( A \sd C ) + \Pr ( B \sd C)$ to get
\begin{align*}
\Pr ( ( A_k \sd E_n ) \cap  \{ N_{2k} \leq n \} ) & \leq \Pr ( A_k \sd E_n ) \\
& \leq \Pr (A_k \sd A ) + \Pr (E_n \sd A ) \leq 2 \eps .\end{align*}
Hence the claim~\eqref{eqn:zeroone2} is verified.
Since $\Pr ( (A \sd B) \cap D ) \leq \Pr ( (A \sd C) \cap D ) + \Pr (   B \sd C  )$,
we also get that
\begin{align*}
\Pr ( ( A \sd A_k' ) \cap \{ N_{2k} \leq n \} ) \leq 
\Pr ( ( A'_k \sd E_n ) \cap \{ N_{2k} \leq n \} ) + \Pr (   A \sd E_n )    
\leq 3\eps,
\end{align*}
by~\eqref{eqn:zeroone2}. Hence
\begin{align*}
\Pr( A \sd A'_k) & \leq \Pr ( N_{2k} > n) + \Pr ( (A \sd A_k' ) \cap \{ N_{2k} \leq n \} ) 
\leq 4\eps .\end{align*}
The final sequence of the proof is  a variation on the standard argument. First note that 
\begin{align}
|\Pr(A)^2-\Pr(A)|  \leq |\Pr(A)^2 - \Pr(A_k \cap A'_k)| + |\Pr(A_k \cap A'_k) - \Pr(A)|. 
\label{eqn:zeroone1}
\end{align}
For the first term on the right-hand side of~\eqref{eqn:zeroone1}, we use the fact that $A_k$ and $A_k'$ are independent with $\Pr(A_k)= \Pr(A'_k)$,
 along with the property of the symmetric difference operator that $|\Pr(A)-\Pr(B)|\leq \Pr(A\sd B)$, to get
\begin{align*}
|\Pr(A)^2-\Pr(A_k \cap A'_k)|  & =   |\Pr(A)^2-\Pr(A_k)^2|\\
& \leq |\Pr(A)+\Pr(A_k)||\Pr(A)-\Pr(A_k)|\\
&\leq 2 \Pr(A \sd A_k) \leq 2\eps.
\end{align*}
Now considering the second term on the right-hand side of~\eqref{eqn:zeroone1} and using the fact that $\Pr(A \sd (B \cap C)) \leq \Pr(A \sd B) + \Pr(A \sd C)$, we have
\begin{align*}
|\Pr(A_k \cap A'_k) - \Pr(A)| &  \leq \Pr(A \sd (A_k \cap A'_k))\\
&\leq \Pr(A \sd A_k) + \Pr( A \sd A'_k) \leq 5 \eps.
\end{align*}
Combining these two bounds, we obtain from~\eqref{eqn:zeroone1} that $|\Pr(A)^2-\Pr(A)| \leq 7 \eps$.
Since $\eps>0$ was arbitrary, we get the result.
\end{proof}

The strategy of the proof of Theorem~\ref{thm:shape}, carried out in the remainder
of this section, is as follows. We use Donsker's theorem and the mapping theorem
to show that $D_n^{-1} \cH_n$ converges weakly to the convex hull of an appropriate Brownian motion, scaled to have unit diameter (Lemma~\ref{lem:shape-weak}).
This limiting set has positive probability of being an arbitrarily good approximation to any given unit-diameter convex compact set $K$.
An application of the zero-one law (Theorem~\ref{thm:zero-one}) then completes the proof.

For $K \in \cK$ let $\cD ( K) := \diam K$.
The next result shows that 
the map
 $K \mapsto \cD(K)$ is continuous 
 from $(\cK ,\rho_H)$ to $(\R_+,\rho_1)$.

\begin{lemma}\label{lem:D-cont}
For $K_1, K_2 \in \cK$, $| \cD ( K_1) - \cD (K_2 ) | \leq 2 \rho_H (K_1, K_2 )$.
\end{lemma}
\begin{proof}
Let $\rho_H ( K_1, K_2 ) = r$. From~\eqref{eq:rhoH} we have that for any $\bx_1, \bx_2 \in K_1$,
there exist $\by_1, \by_2 \in K_2$ such that $\| \bx_i - \by_i \| \leq s$ for any $s > r$. Then,
\[ \| \bx_1 - \bx_2 \| \leq \| \bx_1 - \by_1 \| + \| \by_1 - \by_2 \| + \| \by_2 - \bx_2 \| \leq 2 s + \cD ( K_2 ) .\]
Hence
$\cD(K_1)  \leq 2s + \cD( K_2)$, and since $s >r$ was arbitrary we get $\cD( K_1) - \cD( K_2) \leq 2r$.
A symmetric argument gives  $\cD( K_2) - \cD( K_1) \leq 2r$.
\end{proof}

For $K \in \cK$ and $\bx \in \Sp := \{ \by \in \R^2 : \| \by \| = 1 \}$, define
 $h_K (\bx) := \sup_{\by \in K} (\by \cdot \bx)$. Equivalent to~\eqref{eq:rhoH} for $K_1, K_2 \in \cK$ is the formula~\cite[p.~84]{gruber}
\begin{equation}
\label{eq:rhoH2}
 \rho_H ( K_1, K_2 ) = \sup_{\bx \in \Sp} | h_{K_1} (\bx) - h_{K_2} (\bx) | .\end{equation}
Let $\cK^\star := \{ K \in \cK : \cD ( K) > 0 \} = \cK \setminus \{ \{ \0 \} \}$.

\begin{lemma}
\label{lem:scale-cont}
Suppose that $K_1, K_2 \in \cK^\star$. Then
\begin{equation}
\label{eq:scale-cont} \rho_H ( K_1 / \cD (K_1) , K_2 / \cD (K_2 )) \leq \frac{3 \rho_H (K_1, K_2)}{\cD (K_1 ) } .\end{equation}
In particular, the map $K \mapsto K/\cD(K)$ is   continuous   from $(\cK^\star ,\rho_H)$ to $(\cK^\star ,\rho_H)$. 
\end{lemma}
\begin{proof}
We first claim that for $K_1, K_2 \in \cK$ and $\alpha_1, \alpha_2 >0$,
\begin{equation}
\label{eq:rhoH-scale}
\rho_H ( \alpha_1 K_1 , \alpha_2 K_2 ) \leq \alpha_1 \rho_H ( K_1, K_2 ) + | \alpha_1 - \alpha_2 | \cD ( K_2 ) .\end{equation}
Suppose that $K_1, K_2 \in \cK^\star$. 
Applying~\eqref{eq:rhoH-scale} with $\alpha_i = 1/\cD (K_i)$, we get
\begin{align*}
\rho_H ( K_1 / \cD (K_1) , K_2 / \cD (K_2 )) & \leq \frac{\rho_H ( K_1, K_2)}{\cD (K_1)} + \frac{| \cD ( K_1 ) - \cD (K_2)|}{\cD ( K_1 ) } 
,\end{align*}
from which~\eqref{eq:scale-cont} follows
by Lemma~\ref{lem:D-cont}.  This gives the desired continuity.

It remains to verify the claim~\eqref{eq:rhoH-scale}. From~\eqref{eq:rhoH2}, 
with the observation that, for $\alpha >0$,
 $h_{\alpha K} (\bx) = \alpha h_K (\bx)$, it follows that
\begin{align*} \rho_H ( \alpha_1 K_1 , \alpha_2 K_2 ) & = \sup_{\bx \in \Sp} | \alpha_1 h_{K_1} (\bx) - \alpha_1 h_{K_2} (\bx) + (\alpha_1 - \alpha_2) h_{K_2} (\bx) |\\
& \leq \alpha_1 \sup_{\bx \in \Sp} | h_{K_1} (\bx) - h_{K_2} (\bx) | + | \alpha_1 - \alpha_2 | \sup_{\bx \in \Sp} h_{K_2} (\bx) ,\end{align*}
from which the claim~\eqref{eq:rhoH-scale} follows.
\end{proof}

Suppose that $\Sigma := \Exp ( Z Z^\tra )$ is positive definite.
Let $(b(t), t \geq 0)$ be standard Brownian motion in $\R^2$.
Let $h_1 := \hull b [0,1]$, the convex hull of   Brownian motion run for unit time.
Let $\Sigma^{1/2}$ denote the (unique) positive-definite symmetric matrix such that
$\Sigma^{1/2} \Sigma^{1/2} = \Sigma$.
The map $\bx \mapsto \Sigma^{1/2} \bx$ is an affine transformation of $\R^2$, such that 
$\Sigma^{1/2} b$ is Brownian motion with covariance matrix $\Sigma$,
and $\Sigma^{1/2} h_1 = \hull \Sigma^{1/2} b [0,1]$ is the corresponding convex hull.

\begin{lemma}\label{lem:shape-weak}
Suppose that $\Exp ( \| Z \|^2 ) < \infty$, $\mu = \0$, and $\Sigma$ is positive definite.
Then
\[D_n^{-1}\cH_n \Rightarrow \frac{\Sigma^{1/2}h_1}{\cD(\Sigma^{1/2}h_1)},\]
in the sense of weak convergence on $(\cK, \rho_H)$.
\end{lemma}
\begin{proof}
The convergence $n^{-1/2} \cH_n \Rightarrow \Sigma^{1/2} h_1$ is given in Theorem~2.5 of~\cite{wx2}.
Since (by Lemma~\ref{lem:scale-cont}) $K \mapsto K/\cD(K)$ is continuous on $\cK^\star$,
and $\Pr ( \Sigma^{1/2} h_1 \in \cK^\star ) = 1$, we may apply the mapping theorem~\cite[p.~21]{billc} to deduce the result.
\end{proof}

\begin{proof}[Proof of Theorem~\ref{thm:shape}]
Fix $K \in \cK$ with $\cD ( K) =1$.
We claim that, for any $\eps >0$,
\begin{equation}
\label{shape-claim}
 \Pr \left( \liminf_{n \to \infty} \rho_H \left(D_n^{-1}\cH_n,K\right) \leq \eps \right) >0. \end{equation}
Under the conditions of the theorem, $S_n$ is genuinely 2-dimensional and recurrent~\cite[p.~195]{dur}, and so, by Proposition~\ref{prop:recurrence}, $r_n \to \infty$ a.s.
Since the event in~\eqref{shape-claim} is in $\cT$, the zero-one law (Theorem~\ref{thm:zero-one})
shows that the probability in~\eqref{shape-claim} must be equal to 1. Since $\eps >0$ was arbitrary, the statement of the theorem follows.

Thus it remains to prove the claim~\eqref{shape-claim}. To this end, 
observe that, for any  $\eps >0$,
\begin{align*}
\Pr \left( \liminf_{n \to \infty} \rho_H \left(D_n^{-1}\cH_n,K\right) \leq \eps \right) & \geq \Pr\left(\rho_H\left(D_n^{-1}\cH_n,K\right)<\eps \io \right) \\
&=\Pr\left( \bigcap_{n=1}^\infty \bigcup_{m\geq n} \left\{\rho_H (D_m^{-1}\cH_m,K)<\eps  \right\} \right)\\
&=\lim_{n\to\infty} \Pr \left( \bigcup_{m\geq n} \left\{ \rho_H (D_m^{-1}\cH_m,K)<\eps \right\}\right)\\
&\geq \lim_{n\to \infty} \Pr \left(\rho_H(D_n^{-1}\cH_n,K)<\eps \right).\end{align*}
By the triangle inequality, $| \rho_H ( K , K_1 ) - \rho_H (K , K_2) | \leq \rho_H ( K_1, K_2)$, i.e.,
for fixed $K$,
the function $K_1 \mapsto \rho_H (K, K_1)$ is continuous. Thus by Lemma~\ref{lem:shape-weak}
and the mapping theorem 
\begin{equation}
\label{eq1}
 \lim_{n\to \infty} \Pr \left(\rho_H(D_n^{-1}\cH_n,K)<\eps \right) = 
\Pr \left( \rho_H \left( \frac{\Sigma^{1/2}h_1}{\cD(\Sigma^{1/2}h_1)},K\right) < \eps  \right).
\end{equation}
Let $\delta \in (0, \eps/6)$.
For convenience, set $A = \Sigma^{1/2} h_1$. 
First suppose that $\0$ is in the interior of $K$.
Then, 
it is not hard to see that $K \subseteq A \subseteq (1+\delta ) K$ occurs with positive probability
(one can force the Brownian motion to make a `loop' in $( (1+\delta) K ) \setminus K$). On this event,
we have $h_K (\bx) \leq h_A (\bx) \leq (1+\delta) h_K (\bx)$ for all $\bx \in \Sp$, so that, by~\eqref{eq:rhoH2},
\[ \rho_H (A, K) = \sup_{\bx \in \Sp} | h_A (\bx) - h_K (\bx) | \leq \delta \sup_{\bx \in \Sp} h_K (\bx) \leq \delta \cD (K) = \delta .\]
It follows from taking $K_1 = K$ and $K_2 = A$ in~\eqref{eq:scale-cont}  that
\[ \rho_H ( A / \cD (A) , K ) \leq 3 \rho_H ( A, K ) \leq 3 \delta < \eps/2 .\]
If $\0$ is not in the interior of $K$, then we can find $K' \in \cK$ with $K \subset K'$ such that $\0$
is in the interior of $K'$ and $\rho_H (K, K') < \eps /2$. Then
\[ \rho_H ( A / \cD (A), K) \leq \rho_H (A / \cD (A), K') + \rho_H (K, K') < \eps ,\]
on the event $K' \subseteq A \subseteq (1+\delta ) K'$, which has positive probability.
Hence, in either case, the probability on the right-hand side of~\eqref{eq1} is strictly positive,
establishing~\eqref{shape-claim}.
\end{proof}

\begin{proof}[Proof of Corollary~\ref{cor:infsupLD}]
For $K \in \cK$, let $\cL (K)$ denote the perimeter length of $K$; then, 
Lemma~2.4 of~\cite{wx2} shows that
\begin{equation}
\label{eq:L-cont}
| \cL (K_1 ) - \cL (K_2) | \leq 2 \pi \rho_H (K_1, K_2 ), \text{ for any } K_1, K_2 \in \cK. 
\end{equation}
First, take $K$ to be a  unit-length line segment in $\R^2$ containing $\0$.
Theorem~\ref{thm:shape} shows that, for any $\eps > 0$,
$\rho_H ( D_n^{-1} \cH_n , K ) < \eps$ i.o., a.s. Hence, by~\eqref{eq:L-cont},
\[ L_n / D_n = \cL ( D_n^{-1} \cH_n ) \leq \cL ( K ) + 2 \pi \eps , \io, \]
and $\cL (K ) = 2$. Since $\eps >0$ was arbitrary, we get  $\liminf_{n \to \infty} L_n /D_n \leq 2$,
and  the first inequality in~\eqref{l-d-ineq} shows that this latter inequality is in fact an equality.

Now take $K$ to be a unit-diameter disc in $\R^2$ containing $\0$. Again, Theorem~\ref{thm:shape} shows that, for any $\eps > 0$,
$\rho_H ( D_n^{-1} \cH_n , K ) < \eps$ i.o., a.s.  Hence, by~\eqref{eq:L-cont},
\[ L_n / D_n = \cL ( D_n^{-1} \cH_n ) \geq \cL ( K ) - 2 \pi \eps , \io, \]
and since now $\cL (K) = \pi$ we get $\limsup_{n \to \infty} L_n /D_n \geq \pi$,
which combined with the second  inequality in~\eqref{l-d-ineq} completes the proof.
\end{proof}

\section{Perimeter in the case with drift}
\label{sec:perim-drift}

Suppose that $\Exp ( \| Z\|^2 )< \infty$ and $\mu \neq \0$.
We work towards the proof of Theorem~\ref{thm:drift-mean}.
Write $X_n := S_n \cdot \hat \mu$ and $Y_n := S_n \cdot \hat \mu_\per$, where $\hat \mu_\per$
is any fixed unit vector orthogonal to $\mu$. 
Then $X_n$ and $Y_n$ are one-dimensional random walks with increment
distributions $Z \cdot \hat \mu$ and $Z \cdot \hat \mu_\per$ respectively;
note that $\Exp ( Z \cdot \hat \mu ) = \| \mu \|$, $\Exp ( Z \cdot \hat \mu_\per ) = 0$, $\Var ( Z \cdot \hat \mu ) = \spara$,  and 
\begin{align*} \Var ( Z \cdot \hat\mu_\per ) = \Exp [ ( ( Z -\mu) \cdot \hat \mu_\per )^2 ] & 
 = \Exp [ \| Z - \mu \|^2 ] - \Exp [ ( ( Z-\mu ) \cdot \hat \mu )^2 ] \\
& = \sigma^2 - \spara = \sperp .\end{align*}
The first step towards the proof of Theorem~\ref{thm:drift-mean} is the following result.

\begin{lemma}
\label{lem:ui}
Suppose that $\Exp ( \| Z\|^2 )< \infty$ and $\mu \neq \0$. Then
$\|S_n\| - |S_n \cdot \hat \mu|$ is uniformly integrable.
\end{lemma}
\begin{proof}
The central limit theorem shows that
 $n^{-1} Y_n^2 \tod \sperp \zeta^2$ where $\zeta \sim \cN(0,1)$. Also, since
$\Exp [ Y_n^2 ] =   n \sperp$,
  $n^{-1} \Exp ( Y_n^2 ) \to \sperp = \Exp ( \sperp \zeta^2 )$. 
It is a fact that
if $\theta, \theta_1, \theta_2, \ldots$ are $\RP$-valued random variables with $\theta_n \tod \theta$, then $\Exp  \theta_n  \to \Exp  \theta  < \infty$ if and only if
$\theta_n$ is uniformly integrable: see~\cite[Lemma~4.11]{kall}. Hence we conclude that
\begin{equation}
\label{ui1} n^{-1} Y_n^2 \text{ is uniformly integrable}.
\end{equation}

Fix $\eps >0$. Let $\delta \in (0,\| \mu \|)$ to be chosen later.
For ease of notation, write $T_n =  \| S_n \| - | X_n |$. Then since $T_n \leq \| S_n \|$ and $| X_n | \leq \| S_n \|$, we have
\begin{align*}
\Exp \left[ T_n \1 { T_n > M } \1 { \| S_n \| \leq \delta n } \right] & \leq \delta n \Pr ( \| S_n \| \leq \delta n ) \\
& \leq \delta n \Pr ( | X_n | \leq \delta n ) \\
& \leq \delta n \Pr ( | X_n -   \| \mu \| n | > ( \|\mu \| - \delta ) n ) .\end{align*}
Since $\Exp X_n = n \| \mu\|$ and  $\Var X_n = n \spara$, Chebyshev's inequality then yields
\begin{align*}
 \Exp \left[ T_n \1 { T_n > M } \1 { \| S_n \| \leq \delta n } \right] & \leq \delta n \frac{n \spara}{ ( \| \mu \| - \delta)^2 n^2}  .\end{align*}
It follows that, for suitable choice of $\delta$ (not depending on $M$) and any $M \in (0,\infty)$,
\begin{align*}
\sup_n \Exp \left[ T_n \1 { T_n > M } \1 { \| S_n \| \leq \delta n } \right] \leq \eps. \end{align*}
On the other hand, we use the fact that
\begin{equation}
\label{Ysq}
 0 \leq \| S_n \| - | X_n | = T_n = \frac{ \| S_n \|^2 - X_n^2}{\| S_n \| + | X_n |} = \frac{Y_n^2}{\| S_n \| + | X_n |}   .\end{equation}
Hence 
\begin{align*}
\Exp \left[ T_n \1 { T_n > M } \1 { \| S_n \| > \delta n } \right] &  = \Exp \left[ \tfrac{Y_n^2}{\| S_n \| + | X_n |}  \mathbf{1} \left\{ \tfrac{Y_n^2}{\| S_n \| + | X_n |}  > M \right\} \1 { \| S_n \| > \delta n } \right]\\
& \leq \frac{1}{\delta n} \Exp \left[ Y_n^2 \1 { Y_n^2 > M \delta n } \right]
.\end{align*}
It follows that
\[ \sup_n \Exp \left[ T_n \1 { T_n > M } \1 { \| S_n \| > \delta n } \right]  \leq \frac{1}{\delta} \sup_n \Exp \left[ n^{-1} Y_n^2 \1 { n^{-1} Y_n^2 > M \delta } \right] ,\]
which, for fixed $\delta$, tends to $0$ as $M \to \infty$ by~\eqref{ui1}.

Thus for any $\eps >0$ we have that $\sup_n \Exp \left[ T_n \1 { T_n > M }   \right] \leq \eps$, 
for all $M$ sufficiently large, which completes the proof.
\end{proof}

The next result is of some independent interest, and may be known, although we could find no reference. 

\begin{lemma}
\label{lem2}
Suppose that $\Exp ( \| Z \|^2 ) < \infty$ and $\mu \neq \0$. Then
\[ 0 \leq \| S_n \| - S_n \cdot \hat \mu \to \frac{\sperp \zeta^2}{2 \| \mu \|}, \text{ in } L^1, \text{ as } n \to \infty,\]
for $\zeta \sim \cN(0,1)$. In particular,
\[ 0 \leq  \Exp \| S_n \| - \| \mu \| n  = \frac{\sperp}{2 \| \mu \|} + o(1), \text{ as } n \to \infty .\]
\end{lemma}
\begin{proof}
As above, for $x \in \R$ set $x^+ := x \1 { x >0}$, and also set $x^- = -x \1 { x < 0 }$.
Then $x = x^+ - x^-$ and $|x| = x^+ + x^-$, so $x = |x| -2x^-$;
thus $| X_n | - 2  X_n^- =   X_n \leq  | X_n |$, and  
\begin{equation}
\label{X-bounds}
   0 \leq \| S_n \| - | X_n|  \leq \| S_n \| - X_n = \| S_n \| - |X_n| + 2 X_n^- ;\end{equation}
	in particular $\Exp \| S_n \| \geq \Exp X_n = \| \mu \| n$.
Now, we have from~\eqref{Ysq} that
\[ \| S_n \| - |X_n | = \frac{Y_n^2}{\| S_n \| + | X_n |} = \frac{n^{-1} Y_n^2}{ n^{-1} \| S_n \| + n^{-1} | X_n |} , \]
where $n^{-1} Y_n^2 \tod \sperp \zeta^2$ for $\zeta \sim \cN(0,1)$, and, by the strong law of large numbers,
 both $n^{-1} \| S_n\|$ and $n^{-1} | X_n |$ tend to
$\| \mu \|$ a.s. Hence $0 \leq \| S_n \| - |X_n| \tod \frac{\sperp \zeta^2}{2 \| \mu \|}$, and by Lemma~\ref{lem:ui} we can conclude that
$\| S_n\| - | X_n |  \to \frac{\sperp \zeta^2}{2 \| \mu \|}$ in $L^1$.
Moreover,  Lemma~\ref{lem:negative-part} shows that $X_n^- \to 0$ in $L^1$.
Thus the result follows from~\eqref{X-bounds}.
\end{proof}

We can now complete the proof of Theorem~\ref{thm:drift-mean} and then the proof of Theorem~\ref{thm1}.

\begin{proof}[Proof of Theorem~\ref{thm:drift-mean}.]
From the Spitzer--Widom formula~\eqref{s-w} and Lemma~\ref{lem2}, we have  
\begin{align*} \Exp L_n  = 2 \sum_{k=1}^n \frac{1}{k} \left( \| \mu \| k + \frac{\sperp}{2 \| \mu \|} + o(1) \right)  
  = 2 \|\mu \| n + \frac{\sperp}{\| \mu \|} \log n + o (\log n) ,\end{align*}
as claimed.
\end{proof}
 
\begin{proof}[Proof of Theorem~\ref{thm1}.]
Theorem~\ref{thm:drift-mean} shows that
\begin{equation}
\label{eq2} n^{-1/2} | \Exp L_n - 2 \Exp S_n \cdot \hat \mu | \to 0 .
\end{equation}
Then by the triangle inequality
\[ n^{-1/2} | L_n - 2  S_n \cdot \hat \mu | \leq n^{-1/2} | L_n - \Exp L_n - 2 (S_n - \Exp S_n ) \cdot \hat \mu |
+  n^{-1/2} | \Exp L_n - 2 \Exp S_n \cdot \hat \mu | ,\]
which tends to $0$ in~$L^2$ by~\eqref{wx} and~\eqref{eq2}.
\end{proof}

\section{Diameter in the case with drift}
\label{sec:diam-drift}

Now we turn to the diameter $D_n$.
The main aim of this section is to establish the following result, from which we will deduce Theorem~\ref{thm2}.

\begin{theorem}\label{thm:DnL2}
Suppose that $\Exp ( \| Z  \|^2 ) < \infty$ and $\mu  \neq \0$.  Then, as $n\to\infty$,
\begin{equation}\label{eqn:DnL2conv}
n^{-1/2}\left| D_n - \Exp D_n  - (S_n-\Exp S_n )\cdot \hat \mu \right| \rightarrow 0, \text{ in } L^2.
\end{equation}
\end{theorem}

Theorem~\ref{thm:DnL2} is the analogue for $D_n$
of the result~\eqref{wx} for $L_n$, established in Theorem~1.3 of~\cite{wx1}. Our  approach
to proving Theorem~\ref{thm:DnL2}
is similar in outline to that in~\cite{wx1}, where a martingale difference idea (which we explain below in the present context)
was combined with Cauchy's formula for the perimeter length. Here, the place of Cauchy's formula is taken by 
the formula
\begin{equation}
\label{diam-formula}
 \diam A = \sup_{0 \leq \theta \leq \pi} \rho_A (\theta ), \end{equation}
where $A \subset \R^d$ is a non-empty compact set, and $\rho_A(\theta) := \sup_{\bx \in A} ( \bx \cdot \be_\theta ) - \inf_{\bx \in A} (\bx \cdot \be_\theta)$;
see Lemma~6 of~\cite{mx} for a derivation of~\eqref{diam-formula}.

Before embarking on the proof of Theorem~\ref{thm:DnL2}, we observe the following result.
\begin{lemma}
\label{lem:diam-expec}
Suppose that $\Exp ( \| Z  \|^2 ) < \infty$ and $\mu  \neq \0$. 
There exists  $C < \infty$ such that 
\[ 0 \leq \Exp   D_n -  \| \mu \| n \leq C( 1 + \log n), \text{ for all } n \geq 1. \]
\end{lemma}
\begin{proof}
The lower bound follows from Lemma~\ref{lem2} and the fact that $D_n \geq \| S_n \|$.
The upper bound follows from the fact that $D_n \leq L_n/2$ and the fact that,
by Theorem~\ref{thm:drift-mean}, $\Exp L_n \leq 2\|\mu \| n + C(1+ \log n)$.
\end{proof}

Now we describe the martingale difference construction, which is standard.
Recall that $\cF_0 := \{ \emptyset, \Omega \}$ and $\cF_n := \sigma(Z_1,\ldots,Z_n)$ for $n \geq 1$.
Let $Z_1',Z_2',\ldots$ be an independent copy of the sequence $Z_1,Z_2,\ldots$. 
Fix $n \in \N$. For $i \in \{1,\ldots,n\}$, define
\[
S_j^{(i)} := \begin{cases} S_j &\text{if }\ j<i,\\ S_j-Z_i+Z_i' &\text{if }\ j\geq i; \end{cases} 
\]
then $(S_j^{(i)};0\leq j \leq n)$ is the random walk $(S_j;0\leq j \leq n)$ but with $Z_i$ `resampled' and replaced by $Z_i'$.
For $i \in \{1,\ldots,n\}$, define
\begin{equation}
D_n^{(i)} := \diam \{ S_0^{(i)},\ldots,S_n^{(i)}\}, \text{ and } \Delta_{n,i}:= \Exp ( D_n-D_n^{(i)} \mid \cF_i ).
\label{eqn:Dni}
\end{equation}
Observe that we also have the representation $\Delta_{n,i} = \Exp ( D_n \mid \cF_i ) - \Exp ( D_n \mid \cF_{i-1} )$
and hence $\Delta_{n,i}$ is a martingale difference sequence, i.e., $\Delta_{n,i}$ is $\cF_i$-measurable
with $\Exp ( \Delta_{n,i} \mid \cF_{i-1} ) = 0$.
The utility of this construction is the following result (see e.g.~Lemma~2.1 of~\cite{wx1}).
\begin{lemma} 
\label{lemma:VarGn}
Let $n\in \N$. Then  $D_n - \Exp D_n = \sum_{i=1}^n \Delta_{n,i}$, and  $\Var D_n = \sum_{i=1}^n\Exp ( \Delta_{n,i}^2 )$.
\end{lemma}

Recall that $\unittheta$ denotes the unit vector in direction $\theta$.
For $\theta \in [0,\pi]$, define 
\[
M_n(\theta):= \max_{0\leq j\leq n} (S_j\cdot\unittheta), \text{ and }
 m_n(\theta):= \min_{0\leq j\leq n} (S_j\cdot \unittheta) ,\]
and define   $R_n(\theta) := M_n(\theta)-m_n(\theta)$.
Note that since $S_0=\0$, we have $M_n(\theta)\geq 0$ and $m_n(\theta)\leq 0$, a.s. 
It follows from~\eqref{diam-formula} that $D_n = \sup_{0\leq \theta \leq \pi} R_n(\theta)$.

Similarly, when the $i$th increment is resampled,
$D_n^{(i)} = \sup_{0\leq \theta \leq \pi} R_n^{(i)}(\theta)$,
where $R_n^{(i)}(\theta) := M_n^{(i)}(\theta)-m_n^{(i)}(\theta)$, with
\[M_n^{(i)}(\theta):= \max_{0\leq j\leq n}(S_j^{(i)}\cdot\unittheta), \text{ and } m_n^{(i)}(\theta):= \min_{0\leq j\leq n}(S_j^{(i)}\cdot \unittheta). \]

Thus to study $\Delta_{n,i}$ as defined at~\eqref{eqn:Dni}, we are interested in
\begin{equation}
\label{eq:d-difference}
 D_n - D_n^{(i)} = \sup_{0\leq \theta \leq \pi} R_n(\theta) - \sup_{0\leq \theta \leq \pi} R_n^{(i)}(\theta) .\end{equation}
For the remainder of this section we suppose,
without loss of generality, that $\mu=\|\mu\| \be_{\pi/2}$ with $\|\mu\| \in (0,\infty)$.
An important observation is that the diameter does not
deviate far from the direction of the drift.
For $\delta \in (0,\pi/2)$ and $i \in \{1,\ldots,n\}$, define the event  
\[ A_{n,i} (\delta):= \left\{ \left| \frac{\pi}{2} - \argmax_{0\leq \theta \leq \pi} R_n(\theta) \right| < \delta  \right\}
\cap \left\{ \left| \frac{\pi}{2} - \argmax_{0\leq \theta \leq \pi} R^{(i)}_n(\theta) \right| < \delta  \right\} . \]

\begin{lemma}\label{lemma:Dthetabound}
Suppose that $\Exp\|Z \|<\infty$ and $\mu = \|\mu\|\be_{\pi/2} \neq \0$. 
Then for any $\delta \in (0,\pi/2)$, $\lim_{n\rightarrow \infty} \min_{1 \leq i \leq n} \Pr(A_{n,i}(\delta))=1$.
\end{lemma}
\begin{proof}
Fix $\delta \in (0,\pi/2)$.
Note that $S_j \cdot \be_0$ is a random walk on $\R$ with mean increment $\Exp ( Z \cdot \unitzero ) = \mu \cdot \unitzero = 0$. 
Hence the strong law of large numbers implies that for any $\eps  >0$, 
\[ \max_{0 \leq j \leq n} | S_j \cdot \unitzero | \leq \eps n,  \]
for all $n\geq N_{\eps}$ with $\Pr(N_{\eps}<\infty)=1$.
Similarly, since $S_j \cdot  \unithalfpi$ is a random walk on $\R$
with mean increment $\| \mu \| >0$, there exists $N'$ with $\Pr(N'<\infty)=1$ such that
\[ S_j \cdot  \unithalfpi \geq \tfrac{1}{2} \|\mu\| j, \text{ for all } j \geq N'. \]
Let $A'_n(\eps)$ denote the event
\[\Bigl\{\max_{0\leq j\leq n} |S_j \cdot \unitzero |\leq \eps n\Bigr\} \cap \left\{ S_n \cdot \be_{\pi/2} \geq \tfrac{1}{2} \|\mu\| n \right\}.\]
Then if $A'_n (\eps ) $ occurs, any line segment that achieves the diameter has length at least $\frac{1}{2} \|\mu\|  n$ and horizontal component at most $2\eps  n$. 
Thus if $\theta_n= \argmax_{0\leq \theta \leq \pi} R_n(\theta)$ we have
\[|\cos \theta_n|\leq \frac{4\eps }{\|\mu\|}, \text{ on } A'_n (\eps).\]
Thus for $\eps$ sufficiently small we have that $A'_n(\eps )$ implies $|\theta_n - \pi/2| < \delta$. Hence
\[\Pr( |\theta_n - \pi/2| < \delta )\geq\Pr(A'_n(\eps ))\geq \Pr\left(n\geq \max\{N_\eps ,N'\}\right)\rightarrow \Pr \left(\max\{N_\eps ,N'\}<\infty\right)=1 .\]
But  $\theta^{(i)}_n= \argmax_{0\leq \theta \leq \pi} R^{(i)}_n(\theta)$ has the same distribution as $\theta_n$, so
\[ \min_{1 \leq i \leq n} \Pr ( \{ |\theta_n - \pi/2| < \delta \} \cap \{ |\theta^{(i)}_n - \pi/2| < \delta  \} )
\geq 1 - 2 \Pr ( |\theta_n - \pi/2| \geq \delta ) ,\]
and the result follows.
\end{proof}

Lemma~\ref{lemma:Dthetabound} tells us that the key to understanding~\eqref{eq:d-difference}
is to understand what is happening with  $R_n(\theta)$
and $R_n^{(i)} (\theta)$ for $\theta \approx \pi/2$.
The next important observation is that for $\theta \in (0,\pi)$, the one-dimensional
random walk $S_j \cdot \be_\theta$ has drift $\mu \cdot \be_\theta = \mu \sin \theta >0$,
so, with very high probability $M_n (\theta)$ is attained somewhere near the end of the walk,
and $m_n (\theta)$ somewhere near the start.

To formalize this statement, and its consequence for $R_n(\theta) - R_n^{(i)}(\theta)$, define
\begin{align*}
  \bar{J}_n(\theta) & := \argmax_{0\leq j\leq n} (S_j\cdot \unittheta), \text{ and } \ubar{J}_n(\theta):= \argmin_{0\leq j\leq n} (S_j\cdot \unittheta);\\
	\bar{J}_n^{(i)}(\theta) & := \argmax_{0\leq j\leq n}  (S_j^{(i)}\cdot \unittheta ), \text{ and } 
	\ubar{J}_n^{(i)}(\theta):= \argmin_{0\leq j\leq n}  (S_j^{(i)}\cdot \unittheta ).
	\end{align*}
For $\gamma \in (0,1/2)$ 
 (a constant that will be chosen to be suitably small later in our argument), 
we denote by $E_{n,i}(\gamma)$ the event that the following occur:
\begin{itemize}
\item for all $\theta \in [\pi/4,3\pi/4]$, $\ubar{J}_n(\theta)<\gamma n$ and $\bar{J}_n(\theta)>(1-\gamma) n$;
\item for all $\theta \in [\pi/4,3\pi/4]$, $\ubar{J}_n^{(i)}(\theta)<\gamma n$ and $\bar{J}_n^{(i)}(\theta)>(1-\gamma) n$;
\end{itemize}
note that the choice of interval $[\pi/4, 3 \pi/4]$ could be replaced by any other
interval containing $\pi/2$ and bounded away from $0$ and $\pi$.
Define $I_{n,\gamma}:= \{1, \ldots ,n\}\cap [\gamma n, (1-\gamma)n]$. 
The next result is contained in Lemma~4.1 of~\cite{wx1}.

\begin{lemma}
\label{lem:Eprob}
For any $\gamma \in (0,1/2)$ the following hold.
\begin{itemize}
\item[(i)] If $i \in I_{n,\gamma}$, then, on the event $E_{n,i} (\gamma)$, 
\begin{equation} 
 R_n(\theta) - R_n^{(i)}(\theta) 
= (Z_i - Z'_i)\cdot \unittheta  , \text{ for any } \theta \in [\pi/4, 3\pi/4].
\label{eqn:rangeEquation} 
\end{equation}
\item[(ii)] If $\Exp \| Z \| < \infty$ and $\mu\neq \0$ then $\lim_{n \to \infty} \min_{1\leq i \leq n} \Pr (  E_{n,i}( \gamma) ) = 1$.
\end{itemize}
\end{lemma}

In light of Lemma~\ref{lemma:Dthetabound}, the key to estimating~\eqref{eq:d-difference} is provided by the following.

\begin{lemma}
\label{lem:diambound}
Let $\gamma \in (0,1/2)$. Then for any 
$\delta \in (0, \pi/4)$ and any $i\in I_{n,\gamma}$, on $E_{n,i}(\gamma)$,
\[\left|\sup_{|\theta-\pi/2|\leq \delta}R_n(\theta)-\sup_{|\theta-\pi/2|\leq \delta}R_n^{(i)}(\theta)-(Z_i-Z'_i)\cdot \unithalfpi\right|\leq \delta\|Z_i-Z'_i\|.\]
\end{lemma}

Before proving Lemma~\ref{lem:diambound}, we need a simple geometrical lemma.

\begin{lemma}\label{lemma:RLipschitz}
For any $\bx\in \R^2$ and $\theta_1, \theta_2 \in \R$,
\[|\bx\cdot \be_{\theta_1}-\bx\cdot\be_{\theta_2}|\leq \|\bx\||\theta_1-\theta_2|.\]
\end{lemma}
\begin{proof}
We have
\begin{align*}
\be_{\theta_1}-\be_{\theta_2}&=(\cos \theta_1 - \cos \theta_2, \sin \theta_1 - \sin \theta_2 )\\
&=\left(-2\sin\left( \frac{\theta_1 - \theta_2}{2}\right)\sin\left( \frac{\theta_1 + \theta_2}{2}\right),2\sin\left( \frac{\theta_1 - \theta_2}{2}\right)\cos\left( \frac{\theta_1 + \theta_2}{2}\right)\right),
\end{align*}
so that $\|\be_{\theta_1}-\be_{\theta_2}\|^2 = 4 \sin^2 \left( \frac{\theta_1 - \theta_2}{2}\right)$,
and hence $\|\be_{\theta_1}-\be_{\theta_2}\| = 2 \left| \sin \left( \frac{\theta_1 - \theta_2}{2}\right)\right|$.
Now use the inequality $|\sin x|\leq |x|$ (valid for all $x\in \R$) to get
\[\|\be_{\theta_1}-\be_{\theta_2}\|\leq |\theta_1-\theta_2|,\]
and the result follows.
\end{proof}

\begin{proof}[Proof of Lemma~\ref{lem:diambound}.]
We claim that with $i \in I_{n,\gamma}$, for any $\theta_1,\theta_2 \in [\pi/4,3\pi/4]$, on the event $E_{n,i}(\gamma)$,
it holds that
\begin{equation}
\label{eq:two-theta}
\inf_{\theta_1\leq \theta \leq \theta_2}(Z_i-Z'_i)\cdot \unittheta\leq\sup_{\theta_1\leq \theta \leq \theta_2} R_n(\theta) - \sup_{\theta_1\leq \theta \leq \theta_2} R_n^{(i)}(\theta) \leq \sup_{\theta_1\leq \theta \leq \theta_2}(Z_i-Z'_i)\cdot \unittheta
.
\end{equation}
Given the claim~\eqref{eq:two-theta}, and that, as follows from Lemma~\ref{lemma:RLipschitz},
\begin{align*}
 \sup_{|\theta-\pi/2|\leq \delta}(Z_i-Z'_i)\cdot \unittheta & \leq (Z_i-Z'_i)\cdot \unithalfpi + \delta \|Z_i-Z'_i\|, \text{ and} \\
 \inf_{|\theta-\pi/2|\leq \delta}(Z_i-Z'_i)\cdot \unittheta & \geq (Z_i-Z'_i)\cdot \unithalfpi - \delta \|Z_i-Z'_i\|, \end{align*}
the statement in the lemma follows on taking $\theta_1 = \pi/2 - \delta$ and $\theta_2 = \pi/2 + \delta$.

 It remains to establish the claim~\eqref{eq:two-theta}.
First we note that for $f, g : \R \to \R$ with 
$\sup_{ \theta \in I} |f (\theta) | < \infty$ and $\sup_{ \theta \in I} |g (\theta) | < \infty$,
\begin{equation}\label{eqn:supABineq}
\inf_{\theta \in I} (f(\theta)-g(\theta)) \leq 
\sup_{\theta \in I} f(\theta)-\sup_{\theta \in I} g(\theta)\leq \sup_{\theta \in I} (f(\theta)-g(\theta)).
\end{equation}
In particular, taking $I = [\theta_1, \theta_2]$, with $\theta_1,\theta_2 \in [\pi/3,3\pi/4]$, we have
\begin{align*}
\inf_{\theta_1 \leq \theta \leq \theta_2} \left( R_n (\theta) - R_n^{(i)} (\theta) \right) \leq 
 \sup_{\theta_1\leq \theta \leq \theta_2} R_n(\theta) - \sup_{\theta_1\leq \theta \leq \theta_2} R_n^{(i)}(\theta)  
& \leq \sup_{\theta_1 \leq \theta \leq \theta_2} \left( R_n (\theta) - R_n^{(i)} (\theta) \right) ,
\end{align*}
and, on the event $E_{n,i}(\gamma)$, we have from~\eqref{eqn:rangeEquation}   that
\[ 
  R_n (\theta) - R_n^{(i)} (\theta)  
=   (Z_i-Z'_i) \cdot \unittheta, \text{ for all } \theta \in [\theta_1, \theta_2 ] ,\]
which establishes the claim~\eqref{eq:two-theta}.
\end{proof}

To obtain rough estimates when the events $A_{n,i} (\delta)$ and $E_{n,i} (\gamma)$ do not occur,
we need the following bound. 

\begin{lemma}\label{lemma:ACDeltaBound}
For any $i \in \{ 1,2,\ldots,n\}$, a.s.,
\[
|D_n^{(i)}-D_n| \leq 2\|Z_i\|+2\|Z'_i\|.
\]
\end{lemma}
\begin{proof}
Lemma~3.1 from~\cite{wx1} states that, for any $i \in \{ 1,2,\ldots,n\}$, a.s.,
\[
\sup_{0 \leq \theta \leq \pi} \left| R_n(\theta)- R_n^{(i)}(\theta) \right| \leq 2\|Z_i\| + 2\|Z'_i\|. 
\]
Now from~\eqref{eq:d-difference} and~\eqref{eqn:supABineq} we obtain the result.
\end{proof}

Now define the event $B_{n,i} ( \gamma, \delta):= E_{n,i}( \gamma) \cap A_{n,i}(\delta)$.
Let $B_{n,i}^\rc (\gamma,\delta)$ denote the complementary event.
The preceding results in this section can now be combined to obtain the following approximation lemma
for $\Delta_{n,i}$ as given by~\eqref{eqn:Dni}.
\begin{lemma}
\label{lem:approx}
Suppose that $\Exp\|Z \|<\infty$ and $\mu \neq \0$.
For any $\gamma \in (0,1/2)$, $\delta \in (0,\pi/4)$, and $i \in I_{n,\gamma}$, we have, a.s.,
\begin{align*}
\left| \Delta_{n,i} - (Z_i-\mu)\cdot \hat{\mu}\right| \leq&~ 3 \|Z_i\| \Pr ( B_{n,i}^\rc(\gamma,\delta) \mid \cF_i) + 3\Exp[\|Z'_i\|\2{ B_{n,i}^\rc(\gamma,\delta) } \mid \cF_i ]\nonumber\\
& + \delta \left(\|Z_i\|+\Exp \|Z \| \right).
\end{align*}
\end{lemma}
\begin{proof}
First observe that, since $Z_i$ is $\cF_i$-measurable and $Z_i'$ is independent of $\cF_i$,
\[   \Delta_{n,i}  - (Z_i - \mu ) \cdot \hat \mu = \Exp [ D_n - D_n^{(i)} - (Z_i - Z_i' ) \cdot \hat \mu \mid \cF_i ] .\] 
Hence, by the triangle inequality,
\begin{align*}
|\Delta_{n,i}-(Z_i-\mu)\cdot \hat{\mu}| 
&\leq\Exp\left[ \left| D_n-D_n^{(i)}-(Z_i-Z'_i)\cdot \hat{\mu}\right| \2 { B_{n,i}(\gamma,\delta) } \mid \cF_i\right]\\
&\qquad {} + {} \Exp \left[ \left| D_n-D_n^{(i)}-(Z_i-Z'_i)\cdot \hat{\mu}\right| \2 { B^\rc_{n,i}(\gamma,\delta) } \mid \cF_i\right] .
\end{align*}
Here, by Lemma~\ref{lemma:ACDeltaBound}, we have that
\begin{align*} 
 \Exp\left[\left| D_n-D_n^{(i)}-(Z_i-Z'_i)\cdot \hat{\mu}\right| \2 { B^\rc_{n,i}(\gamma,\delta) } \mid \cF_i\right] 
& \leq 3
\Exp \left[ ( \| Z_i \| + \| Z_i ' \| ) \2 { B^\rc_{n,i}(\gamma,\delta) } \mid \cF_i\right] 
.\end{align*}
Now, on $A_{n,i} (\delta)$ we have that
\[ D_n = \sup_{| \theta - \pi/2 | \leq \delta} R_n (\theta), \text{ and } D_n^{(i)} =\sup_{| \theta - \pi/2 | \leq \delta} R^{(i)}_n (\theta) ,\]
and hence, by Lemma~\ref{lem:diambound}, on $A_{n,i} (\delta) \cap E_{n,i} (\gamma)$, 
\[ | D_n - D_n^{(i)} - (Z_i - Z_i') \cdot \hat \mu | \leq \delta \| Z_i - Z_i' \| .\]
Hence
\[ 
\Exp\left[ \left| D_n-D_n^{(i)}-(Z_i-Z'_i)\cdot \hat{\mu}\right| \2 { B_{n,i}(\gamma,\delta) } \mid \cF_i\right]
\leq \delta \Exp [   \| Z_i \| + \| Z_i' \|  \mid \cF_i ] .\]
Combining these bounds, and using the
fact that $Z_i$ is $\cF_i$-measurable and $Z_i'$ is independent of $\cF_i$,
we obtain the result.
\end{proof}

We are now almost ready to complete the proof of Theorem~\ref{thm:DnL2}.
To do so, we present an analogue of Lemma~6.1 from~\cite{wx1};
we set $V_i := (Z_i-\mu)\cdot \hat{\mu}$, and $W_{n,i}:= \Delta_{n,i}- V_i$. 
\begin{lemma}\label{lem:ExpWsqconv}
Suppose that $\Exp ( \|Z \|^2 ) <\infty$ and $\mu\neq\0$. Then
\[ \lim_{n\rightarrow \infty} n^{-1} \sum_{i=1}^n \Exp ( W_{n,i}^2 )=0 .\]
\end{lemma}
\begin{proof}
The proof is similar to that of Lemma~6.1 of~\cite{wx1}.
Fix $\eps \in (0,1)$. Take $\gamma \in (0,1/2)$ and $\delta \in (0,\pi/4)$, to be specified later. 
Note that from Lemma~\ref{lemma:ACDeltaBound} we have
$| W_{n,i} | \leq 3(\|Z_i\|+\Exp\|Z \|)$,
so that,  provided $\Exp( \|Z\|^2 ) < \infty$, we have $\Exp ( W_{n,i}^2 ) \leq C_0$ for all $n$ and all $i$, for some constant $C_0<\infty$, depending only on the distribution of $Z$.
Hence
\[\frac{1}{n}\sum_{i\not\in I_{n,\gamma}}\Exp ( W_{n,i}^2 ) \leq 2\gamma C_0 .\]
  From now on choose and fix $\gamma >0$ small enough so that $2\gamma C_0< \eps$.
	
Now consider $i\in I_{n,\gamma}$. For such $i$,
Lemma~\ref{lem:approx} yields
an upper bound for $|W_{n,i}|$.
Note that, for any $C_1 < \infty$, since $Z_i'$ is independent of $\cF_i$,
\[ \Exp[\|Z'_i\|\2{B_{n,i}^\rc(\gamma,\delta)} \mid \cF_i]
\leq \Exp [ \|Z \|\1 { \|Z\| \geq C_1 } ]
 + C_1 \Pr[ B_{n,i}^\rc(\gamma,\delta ) \mid \cF_i].\]
Given $\eps \in (0,1)$ we can take $C_1 = C_1 (\eps)$ large enough such that 
$\Exp [ \|Z \|\1 { \|Z\| \geq C_1 } ] \leq \eps$,  by dominated convergence;
for convenience we take $C_1 > 1$ and $C_1 > \Exp \| Z \|$.
 Hence from Lemma~\ref{lem:approx}  we obtain
\begin{align*}
|W_{n,i}|&\leq 3 (\|Z_i\|+C_1)\Pr[B_{n,i}^\rc(\gamma, \delta) \mid \cF_i] + 3 \eps  + \delta \left(\|Z_i\|+\Exp \|Z\| \right) .
\end{align*}
Using the fact that $\Pr[B_{n,i}^\rc(\gamma, \delta) \mid \cF_i] \leq 1$, $\eps \leq 1$, $\delta \leq 1$, and $C_1 > 1$, $C_1 > \Exp \| Z \|$,
we can square both sides of the last display and collect terms to obtain
\begin{align*}
W_{n,i}^2 \leq 27 C_1^2 ( 1 + \| Z_i \|  )^2 \Pr[B_{n,i}^\rc(\gamma, \delta) \mid \cF_i]  + 9 \eps + 13 C_1^2 \delta \left(1 + \|Z_i\|\right)^2 .
\end{align*}
Since $\Exp ( \| Z \|^2 ) < \infty$, it follows that, given $\eps$ and hence $C_1$,
we can choose $\delta \in (0,\pi/4)$ sufficiently small so that $13 C_1^2 \delta \Exp [ \left(1 + \|Z_i\|\right)^2 ] < \eps$;
fix such a $\delta$ from now on.
Then
\[ \Exp ( W_{n,i}^2  ) \leq 27 C_1^2 \Exp [ ( 1 + \| Z_i \|  )^2 \Pr[B_{n,i}^\rc(\gamma, \delta) \mid \cF_i] ] + 10 \eps .\]
 Here we have that, for any $C_2 > 0$,
\[ \Exp [ ( 1 + \| Z_i \|  )^2 \Pr[B_{n,i}^\rc(\gamma, \delta) \mid \cF_i] ] \leq (1 + C_2 )^2 \Pr ( B_{n,i}^\rc(\gamma, \delta) )
+ \Exp [ ( 1 + \| Z  \|  )^2 \1 { \| Z \| \geq C_2 } ] ,\]
where dominated convergence shows that we may choose $C_2$ large enough so that the last term is less than $\eps / C_1^2$, say. Then,
\[ \Exp ( W_{n,i}^2  ) \leq 37 \eps + 27 C_1^2 (1 + C_2)^2  \Pr ( B_{n,i}^\rc(\gamma, \delta) ) .\]
Finally, we see from Lemmas~\ref{lemma:Dthetabound} and~\ref{lem:Eprob} than
$\max_{1 \leq i \leq n} \Pr ( B_{n,i}^\rc(\gamma, \delta) ) \to 0$,
so that, for given $\eps >0$ (and hence $C_1$ and $C_2$) we may choose $n \geq n_0$ sufficiently large so that
$\max_{i \in I_{n,\gamma}} \Exp ( W_{n,i}^2 ) \leq 38 \eps$.
Hence
\[ \frac{1}{n} \sum_{i\in I_{n,\gamma}} \Exp ( W_{n,i}^2 ) \leq 38 \eps,\]
for all $n\geq n_0$. Combining this result with the estimate for $i\not\in I_{n,\gamma}$, we see that
\[\frac{1}{n}\sum_{i=1}^n \Exp( W_{n,i}^2 ) \leq  39 \eps ,\]
for all $n\geq n_0$. Since $\eps>0$ was arbitrary, the result follows.
\end{proof}

\begin{proof}[Proof of Theorem \ref{thm:DnL2}]
First note that $W_{n,i}$ is $\cF_i$-measurable with
$\Exp ( W_{n,i} \mid \cF_{i-1} ) = \Exp ( \Delta_{n,i} \mid \cF_{i-1} ) - \Exp V_i = 0$,
so that $W_{n,i}$ is a martingale difference sequence.  Therefore by orthogonality, 
$n^{-1} \Exp[ (\sum_{i=1}^n W_{n,i})^2 ] = n^{-1}\sum_{i=1}^n \Exp ( W_{n,i}^2 ) \rightarrow 0$ as $n\rightarrow \infty$, by Lemma~\ref{lem:ExpWsqconv}. In other words, 
$n^{-1/2}\sum_{i=1}^n W_{n,i} \rightarrow 0$ in $L^2$. But, by Lemma~\ref{lemma:VarGn},
\[ \sum_{i=1}^n W_{n,i} = \sum_{i=1}^n \Delta_{n,i} - \sum_{i=1}^n ( Z_i - \mu) \cdot \hat \mu = D_n - \Exp D_n - ( S_n - \Exp S_n ) \cdot \hat \mu .\]
This yields the statement in the theorem.
\end{proof}

 Finally we can give the proof of Theorem~\ref{thm2}.
\begin{proof}[Proof of Theorem~\ref{thm2}.]
Lemma~\ref{lem:diam-expec} shows that
\begin{equation}
\label{eq4} n^{-1/2} | \Exp D_n - \Exp S_n \cdot \hat \mu | \to 0 .\end{equation}
Then by the triangle inequality
\[ n^{-1/2} | D_n -  S_n \cdot \hat \mu | \leq n^{-1/2} | D_n - \Exp D_n - (S_n - \Exp S_n ) \cdot \hat \mu |
+  n^{-1/2} | \Exp D_n - \Exp S_n \cdot \hat \mu | ,\]
which tends to $0$ in~$L^2$ by~\eqref{eqn:DnL2conv} and~\eqref{eq4}.
\end{proof}

\begin{proof}[Proof of Corollary~\ref{cor:diam-clt}]
Corollary~\ref{cor:diam-clt} is deduced from Theorem~\ref{thm2} in a very similar
manner to how Theorems~1.1 and~1.2 in \cite{wx1}
were deduced from Theorem~1.3 there, so we omit the details.
\end{proof}

\section{Diameter in the degenerate case}
\label{sec:diam-degen}

The aim of this section is to prove  Theorem~\ref{thm3};
thus we assume $\mu \neq \0$.
First we state a result that
 will enable us to obtain the second statement in Theorem~\ref{thm3} from the first.

\begin{lemma}\label{DnUnifIntLemma}
	Suppose that $\Exp ( \|Z\|^p ) < \infty$ for some $p>4$, $\mu \neq \0$, and $\spara =0$. Then $(D_n-\|\mu\|n)^2$ is uniformly integrable.
\end{lemma}

As in Section~\ref{sec:perim-drift}, we write
 $X_n := S_n \cdot \hat \mu$ and $Y_n := S_n \cdot \hat \mu_\per$, where $\hat \mu_\per$
is any fixed unit vector orthogonal to $\mu$. Note that if $\spara = 0$, then $X_n = n \| \mu \|$ is deterministic.

\begin{proof}[Proof of Lemma~\ref{DnUnifIntLemma}.]
For $i \leq j$, we have $\| S_j - S_i \|^2 = (Y_j - Y_i)^2 + (X_j- X_i)^2$, so that
	\begin{align*}
	(D_n - \|\mu\|n)^2 &= \left( \max_{0\leq i \leq j \leq n} \left( (Y_j - Y_i)^2 + \|\mu\|^2 (j-i)^2\right)^{1/2} -\|\mu\|n \right) ^2\\
	&\leq \left( \|\mu\|n \max_{0\leq i \leq j \leq n} \left( 1 + \frac{(Y_j-Y_i)^2}{\|\mu\|^2 n^2} \right) ^{1/2} - \|\mu\|n\right)^2.
	\end{align*}
Since
	$(1 + y)^{1/2} \leq 1 + (y/2)$ for $y \geq 0$, and $(a-b)^2 \leq 2 (a^2 + b^2)$ for $a, b \in \R$, we obtain
	\begin{align*} 
(D_n - \|\mu\|n)^2  
	\leq \left( \|\mu\|n \max_{0\leq i \leq j \leq n}   \frac{(Y_j-Y_i)^2}{2 \|\mu\|^2 n^2}    \right)^2  \leq \frac{4}{\| \mu \|^{2}} \max_{1\leq i\leq n} \frac{Y_i^4}{n^2} .
	\end{align*}
Now, $|Y_n|$ is a non-negative submartingale, so Doob's~$L^p$ inequality~\cite[p.~505]{gut} yields
\begin{align*}
\Exp \left[ \left( \max_{1\leq i\leq n} \frac{Y_i^4}{n^2} \right)^{p/4} \right]
= n^{-p/2} \Exp \left( \max_{1\leq i\leq n} |Y_i|^p \right) \leq C_p n^{-p/2} \Exp ( |Y_n|^p ), 
\end{align*}
for any $p>1$ and some constant $C_p < \infty$. Assuming that $\Exp ( \| Z \|^p ) < \infty$
for $p >4$,    $Y_n$ is a random walk on $\R$ whose increments have zero mean
and finite $p$th moments, so, by the Marcinkiewicz--Zygmund inequality~\cite[p.~151]{gut},
$\Exp (|Y_n|^p ) \leq C n^{p/2}$. Hence
\[ \sup_{n \geq 0} \Exp \left[ \left( (D_n - \|\mu\|n)^2 \right)^{p/4} \right] < \infty ,\]
which, since $p/4 > 1$, establishes uniform integrability.
\end{proof}
 
Next we show that, under the conditions of  Theorem~\ref{thm3},
the diameter must be attained by a point `close to' the start and one `close to' the end of the walk. 

\begin{lemma}\label{lem:Dattained}
	Suppose that $\Exp ( \|Z\|^2 ) < \infty$, $\mu \neq \0$, and $\spara =0$. Let $\beta \in (0,1)$.
	Then, a.s.,
	for all but finitely many $n$,
	\[ D_n = \max_{\substack{0\leq i \leq n^{\beta} \\ n-n^\beta \leq j \leq n}}\|S_j-S_i\| .\]
	\end{lemma}
	\begin{proof}
Fix $\beta\in (0,1)$. Since $D_n =  \max_{0\leq i,j \leq n}\|S_j-S_i\|$, we have
	\begin{align}
	D_n = \max \left\{ \max_{\substack{0\leq i \leq n^{\beta} \\ n-n^\beta \leq j \leq n}}\|S_j-S_i\|,
	\max_{\substack{0\leq i \leq n^{\beta} \\ 0 \leq j \leq n-n^\beta}}\|S_j-S_i\|,
	\max_{n^{\beta}\leq i,j \leq n}\|S_j-S_i\|\right\}.
	\label{eqn1}
	\end{align}
	It is clear that 
	\[\max_{\substack{0\leq i \leq n^{\beta} \\ n-n^\beta \leq j \leq n}}\|S_j-S_i\| \geq \|S_n\| \geq | X_n | = \|\mu\|n .\]
	We aim to show that the other two terms on the right-hand side of~\eqref{eqn1} are strictly less than 
	$\|\mu\|n$ for all but finitely many $n$. 
	
	A consequence of the law of the iterated logarithm is that, for any $\eps >0$, a.s., for all but finitely many $n$,
	$\max_{0\leq i \leq n} Y_i^2 \leq n^{1+\eps}$; see e.g.~\cite[p.~384]{gut}.
Take $\eps \in (0,\beta)$. Then, 
	\begin{align*}
	\max_{\substack{0\leq i \leq n^{\beta} \\ 0 \leq j \leq n-n^\beta}}\|S_j-S_i\|^2 & \leq \max_{\substack{0\leq i \leq n^{\beta} \\ 0 \leq j \leq n-n^\beta}}|X_j-X_i|^2 +\max_{\substack{0\leq i \leq n^{\beta} \\ 0 \leq j \leq n-n^\beta}}|Y_j - Y_i|^2  \\
	&\leq  \|\mu\|^2 (n-n^\beta)^2 +\max_{0 \leq j \leq n-n^\beta } Y_j^2 +\max_{0\leq i \leq n^\beta} Y_i^2 + 2\max_{\substack{0\leq i \leq n^{\beta} \\ 0 \leq j \leq n-n^\beta}}|Y_j| |Y_i|  \\
	&\leq \|\mu\|^2 n^2 - 2\|\mu\|^2 n^{1+\beta} + \|\mu\|^2 n^{2\beta}+ n^{1+\eps} ,\end{align*}
	for all but finitely many $n$. Since $\eps < \beta < 1$, this last expression
	is strictly less than $\|\mu\|^2 n^2$ for all $n$ sufficiently large.
 Similarly,
	\begin{align*}
	\max_{n^\beta \leq i,j \leq n}\|S_j-S_i\|^2 & \leq 
	\|\mu\|^2 (n-n^\beta)^2 +\max_{n^\beta \leq j \leq n } Y_j^2 +\max_{n^\beta \leq i \leq n} Y_i^2 + 2\max_{n^\beta \leq i,j \leq n}|Y_j| |Y_i|  \\
	&\leq  \|\mu\|^2 n^2 - 2\|\mu\|^2 n^{1+\beta} + \|\mu\|^2 n^{2\beta}+ n^{1+\eps} ,\end{align*}
	for all but finitely many $n$, and, as before, this is strictly less than $\|\mu\|^2 n^2$ for all $n$ sufficiently large.
	Then~\eqref{eqn1} yields the result.
\end{proof}

The main remaining step in  the proof of Theorem~\ref{thm3} is the following result.

\begin{lemma}\label{lem:Dconverge}
	Suppose that $\Exp ( \|Z\|^p )  < \infty$ for some $p>2$, $\mu \neq \0$, and $\spara =0$.  
	Then, as $n \to \infty$, $D_n - \| S_n \| \to 0$, a.s.
	\end{lemma}
	\begin{proof}
Using the fact that $\| S_n \|^2 = \| \mu \|^2 n^2 + Y_n^2$, 	we have that,  for $j \leq n$,
	\begin{align*}
 \|S_j-S_i\|^2  & = \| \mu \|^2 (j - i)^2 + (Y_j - Y_i )^2 \\
& = \| S_n \|^2 + \|\mu\|^2i^2 +\|\mu\|^2j^2 -2\|\mu\|^2ij -\|\mu\|^2n^2  
  + Y_i^2 +Y_j^2 - 2Y_i Y_j -Y_n^2 \\
	&\leq   \|S_n\|^2+ \|\mu\|^2i^2   -(Y_n-Y_j)(Y_n+Y_j)+2Y_i(Y_n-Y_j)-2Y_iY_n + Y_i^2 . \end{align*}
Here we have that, for any $\eps >0$,
 $\max_{0 \leq i \leq n^\beta} | Y_i Y_n | \leq n^{\frac{1+\beta}{2} + \eps}$
and $\max_{0 \leq i \leq n^\beta} Y_i^2 \leq n^{\beta + \eps}$ for all but finitely many $n$.
For the terms involving $Y_j$,  Lemma~\ref{LILStyleLem} shows that we may choose $\beta \in (0,1/2)$ such that,
for any sufficiently small $\eps>0$,
\[ \max_{n-n^\beta \leq j \leq n} |Y_n-Y_j| \leq n^{\frac{1}{2}-\eps} , \text{ and } \max_{n-n^\beta \leq j \leq n} |Y_n-Y_j| |Y_n+Y_j| \leq n^{1-\eps} , \]
for all but finitely many $n$. With this choice of $\beta$ and sufficiently small $\eps$, we combine these bounds to obtain
\[ \max_{\substack{0\leq i \leq n^{\beta} \\ n-n^\beta \leq j \leq n}} \|S_j-S_i\|^2
\leq \| S_n \|^2 + \| \mu \|^2 n^{2\beta} +n^{1-\eps} + n^{\frac{1+\beta}{2} + \eps} + n^{\beta + \eps} , \]
	for all but finitely many $n$. Since $\beta \in (0,1/2)$, we may apply Lemma~\ref{lem:Dattained} and choose $\eps >0$
	sufficiently small to see that
	$D_n^2 \leq \| S_n \|^2 + n^{1-\eps}$,
	for all but finitely many $n$. Hence
	\begin{align*}
	D_n  \leq 
	\|S_n\| \left(1 + \|S_n\|^{-2} n^{1-\eps} \right)^{1/2}  \leq \|S_n\| \left(1+ \|\mu\|^{-2} n^{-1-\eps} \right)^{1/2} ,\end{align*}
	since $\|S_n\|\geq n\|\mu\|$. Using the fact that $(1+ x)^{1/2} \leq 1 + (x/2)$ for $x \geq 0$, we get
	\[ D_n \leq \|S_n\| \left(1+  \tfrac{1}{2} \|\mu\|^{-2} n^{-1-\eps }\right) \leq \|S_n\| +   \|\mu\|^{-1} n^{-\eps} ,\]
	for all but finitely many $n$, since, by the 
	strong law of large numbers, $\|S_n\|\leq 2\|\mu\|n$ all but finitely often. 
	Combined with the bound $D_n \geq \| S_n \|$, this completes the proof.
	\end{proof}

\begin{proof}[Proof of Theorem~\ref{thm3}.]
Combining Lemmas~\ref{lem:Dconverge} and~\ref{lem2} with Slutsky's theorem~\cite[p.~249]{gut}
and the fact that, in this case,  $X_n = \| \mu \| n$, we obtain~\eqref{eqn4}.
	
	From Lemma~\ref{DnUnifIntLemma} we have that, if
  $\Exp( \|Z\|^p ) < \infty$ for $p>4$, both $D_n - \| \mu \| n$ and
	$(D_n-\|\mu\|n)^2$  are uniformly integrable. Thus from~\eqref{eqn4} we obtain
	\begin{align*} \lim_{n \to \infty} \Exp ( D_n-\|\mu\|n ) & = \Exp\left[\frac{\sperp \zeta^2}{2\|\mu\|} \right] = \frac{\sperp}{2\|\mu\|}, \text{ and }\\
	\lim_{n \to \infty} \Exp [(D_n-\|\mu\|n)^2] & = \Exp\left[ \frac{\sigma_{\mu_\per}^4\zeta^4}{4\|\mu\|^2}\right]=\frac{3\sigma_{\mu_\per}^4}{4\|\mu\|^2}.\end{align*}
	Using the fact that 
	\begin{align*}
	\Var D_n  = \Var (D_n -\|\mu\|n) = \Exp[(D_n - \|\mu\|n)^2]-\Exp[D_n - \|\mu\|n]^2 ,\end{align*}
	we obtain~\eqref{eqn5} on letting $n \to \infty$.
\end{proof}

\appendix

\section{Auxiliary results}
\label{sec:appendix}

In this appendix we present two technical results on sums of i.i.d.~random variables
that are needed in the body of the paper. The first is used in the proof
of Lemma~\ref{lem2}.

\begin{lemma}
\label{lem:negative-part}
Let $\xi, \xi_1, \xi_2, \ldots$ be i.i.d.~random variables with  $\Exp ( \xi^2 ) < \infty$ and
$\Exp \xi >0$. Let $X_n = \sum_{k=1}^n \xi_k$. Then $\lim_{n \to \infty} \Exp   X_n^-   =0$.
\end{lemma}
\begin{proof}
Let $\Exp \xi = m >0$ and $\Var \xi = s^2 < \infty$.
Fix $\eps >0$.
Note that
\[ \Exp  X_n^- = \int_0^\infty \Pr ( X_n^- > r ) \ud r
= \int_0^{\eps n} \Pr ( X_n^- > r ) \ud r + \int_{\eps n}^\infty \Pr ( X_n^- > r ) \ud r
.\]
Here we have that, by Chebyshev's inequality,
\[ \Pr (X_n^- > r ) \leq \Pr ( | X_n - m n | > m n + r ) \leq \frac{\Var X_n}{(m n + r)^2}
= \frac{s^2 n}{(m n + r)^2} .\]
It follows that
\begin{equation}
\label{eq20}
\int_0^{\eps n} \Pr ( X_n^- > r ) \ud r \leq s^2 n \int_0^{\eps n} \frac{\ud r}{(m n + r)^2}
\leq \frac{s^2 \eps}{m^2}  .\end{equation}

For $B \in (0,\infty)$
 let $\xi'_k := \xi_k \1 { | \xi_k | \leq B }$
and $\xi''_k := \xi_k \1 { | \xi_k | > B }$.
Set $X_n' := \sum_{k=1}^n \xi'_k$ and $X_n'' := \sum_{k=1}^n \xi''_k$. 
By dominated convergence, we have that as $B \to \infty$,
$\Exp \xi'_1 \to m$, $\Var \xi'_1 \to s^2$, $\Exp | \xi''_1 | \to 0$, and $\Var \xi''_1 \to 0$,
so in particular we may (and do) choose $B$ large enough so that
$\Exp \xi'_1 > m /2$, $\Exp | \xi''_1 | < \eps/4$,
and $\Var \xi''_1 < \eps^2$.

Since $X_n = X_n' + X_n''$, for any $r >0$ we have
\begin{equation}
\label{eq21}
 \Pr ( X_n < - r ) \leq \Pr ( X_n' < -r/2) + \Pr (X_n'' < - r/2 ) .\end{equation}
Here since $\Exp( (\xi'_k)^4 ) \leq B^4 < \infty$ it follows from Markov's inequality and 
the Marcinkiewicz--Zygmund inequality~\cite[p.~151]{gut} that for some constant $C < \infty$
(depending on $B$),
\[ \Pr ( X_n ' < - r ) \leq \Pr ( | X_n' - \Exp X_n'|^4 > ( \Exp X_n' + r)^4 ) \leq
\frac{C n^2}{((m/2) n + r )^4 } .\]
So
\begin{equation}
\label{eq22}
 \int_{\eps n}^\infty  \Pr ( X_n ' < - r/2 ) \ud r \leq 16 C n^2 \int_0^\infty \frac{\ud r}{(m n+r )^4} = O (1/n) .
\end{equation}
On the other hand, by Chebyshev's inequality, for $r > (\eps/4) n$,
\[ \Pr ( X_n'' < - r ) \leq \Pr ( | X_n'' - \Exp X_n'' | > \Exp X_n'' + r ) \leq \frac{\Var X_n''}{(r - (\eps/4) n)^2} \leq \frac{\eps^2 n}{(r - (\eps/4) n)^2} .\]
Hence
\begin{equation}
\label{eq23}
 \int_{\eps n}^\infty  \Pr ( X_n'' < - r/2 ) \leq 4 \eps^2 n \int_{\eps n}^\infty \frac{ \ud r}{(r - (\eps/2) n )^2} = 8 \eps .\end{equation}
So from~\eqref{eq21} with~\eqref{eq22} and~\eqref{eq23}, we have
\[ \limsup_{n \to \infty} \int_{\eps n}^\infty \Pr ( X_n < - r ) \ud r \leq 8 \eps, \]
which combined with~\eqref{eq20} implies that
\[ \limsup_{n \to \infty} \Exp  X_n^- \leq \frac{s^2 \eps}{m^2} + 8 \eps .\]
Since $\eps>0$ was arbitrary, the result follows.
\end{proof}

The next result is used in the proof of Lemma~\ref{lem:Dconverge}.

\begin{lemma}\label{LILStyleLem}
Let $\xi, \xi_1, \xi_2, \ldots$ be i.i.d.~random variables with  $\Exp ( |\xi|^p ) < \infty$ for some $p>2$, and
$\Exp \xi =0$.
	For $0 \leq j \leq n$, let $T_{n,j}:=\sum_{i=n-j}^{n} \xi_i$.
	Then there exist $\beta_0 \in (0,1/2)$ and $\eps_0 \in (0,1/2)$ such that for any $\beta \in (0,\beta_0)$ and any $\eps  \in (0,\eps_0)$,
	\[   \lim_{n\rightarrow \infty}\max_{0\leq j \leq n^\beta} \frac{|T_{n,j}|}{n^{(1/2) - \eps}}=0 , \as \]
\end{lemma}

\begin{remark}\label{p2condition}
On first sight,  by the fact that there are $O(n^\beta)$ terms in the sum $T_{n, j}$,
one's intuition may be misled to conclude that $T_{n,j}$ should be only of size about $n^{\beta/2}$.
However, note that assuming only $\Exp ( \xi^2 ) < \infty$,
 $\max_{0 \leq i \leq n} \xi_i$ can be essentially as big as $n^{1/2}$,
and with probability at least $1/n$ this maximal value is a member of $T_{n,j}$,
and so it seems  reasonable to expect that $T_{n,j}$ should be as big as $n^{1/2}$ infinitely often.
Thus our $p>2$ moments condition seems to be necessary.
\end{remark}
\begin{proof}
	Let $\xi'_i = \xi_i \1{|\xi_i|\leq i^{1/2-\delta}}$ and $\xi''_i = \xi_i \1{|\xi_i| > i^{1/2-\delta}}$ for some $\delta \in (0,1/2)$ to be chosen later.
	Then we use the subadditivity of the supremum, the triangle inequality, and the condition $\eps \in (0,\eps_0)$ to get
	\begin{equation} 
	\max_{0\leq j \leq n^\beta} \frac{|T_{n,j}|}{n^{1/2-\eps}}\leq
 \max_{0\leq j \leq n^\beta} \frac{|\sum_{i=n-j}^{n}(\xi'_i-\Exp\xi'_i)|}{n^{1/2-\eps}}+ 
 \frac{\sum_{i=n- n^\beta}^{n}|\Exp \xi'_i|}{n^{1/2-\eps_0}} + \frac{\sum_{i=n-  n^\beta}^{n}|\xi''_i|}{n^{1/2-\eps_0}},
	\label{eqn:1}
	\end{equation}
	where, and for the rest of this proof, if $n^\beta$ appears in the index of a sum, we understand it to be shorthand for $\lfloor n^\beta \rfloor$.
	By Markov's inequality, since $\Exp ( |\xi|^p )<\infty$ for $p>2$ we have
	\[ \Pr \left( |\xi_i|>i^{1/2-\delta}\right)\leq \frac{\Exp ( |\xi|^p )}{i^{(1/2-\delta)p}} = O(i^{\delta p-p/2}) .\]
	Suppose that  $\delta \in (0 , (p-2)/2p )$, so that $\delta p-p/2<-1$, and thus the Borel--Cantelli lemma implies that
	$\xi''_i=0$ for all but finitely many $i$. Thus, for any $\beta, \eps_0 \in (0,1/2)$,
	\[  \lim_{n\rightarrow \infty} \frac{\sum_{i=n-n^\beta}^{n}|\xi''_i|}{n^{1/2-\eps_0}}=0 , \as \]
		For the second term on the right-hand side of~\eqref{eqn:1},  $\Exp \xi =0$ implies $\left|\Exp \xi'_i\right| = \left|\Exp \xi''_i\right|$, so 
	\begin{equation*}
	\sum_{i=n-n^\beta}^{n} \left|\Exp \xi'_i\right| = 
	\sum_{i=n-n^\beta}^{n} \left|\Exp \xi''_i\right| 
	\leq ( n^\beta + 1) \Exp \left(|\xi|\1{|\xi|> (n/2)^{1/2-\delta}}\right),
	\end{equation*}
	for all $n$ large enough so that $n - n^\beta > n/2$.
	Here
	\[ \Exp \left(|\xi|\1{|\xi|> (n/2)^{1/2-\delta}}\right) = \Exp \left(|\xi|^{2} | \xi |^{-1} \1{|\xi|> (n/2)^{1/2-\delta}}\right)
	\leq C n^{\delta - 1/2} ,\]
	for some constant $C$ depending only on $\Exp ( \xi^2)$.
	Suppose that $\delta \leq 1/4$. Then
	we get $\sum_{i=n-n^\beta}^{n} \left|\Exp \xi'_i\right| = O ( n^{\beta -1/4} )$,
	so that, for any $\beta \in (0,1/2)$ and $\eps_0 \in (0,1/4)$,
	\[\lim_{n \rightarrow \infty} \frac{\sum_{i=n-n^\beta}^{n}|\Exp\xi'_i|}{n^{1/2-\eps_0}} = 0, \as \]
	Finally, we consider the first term on the right-hand side of~\eqref{eqn:1},
	with the truncated, centralised sum, which we denote as $T'_{n,j} := \sum_{i=n-j}^n (\xi'_i-\Exp\xi'_i)$. 
	The $\xi'_i - \Exp \xi_i'$ are independent, zero-mean random variables
	with $| \xi'_i - \Exp \xi_i' | \leq 2 n^{1/2-\delta}$ for $i \leq n$, so we may apply the Azuma--Hoeffding inequality~\cite[p.~33]{penrose}
 to obtain, for any $t \geq 0$,
	\[ \Pr\left(|T'_{n,j}|\geq t \right)\leq 2 \exp\left( - \frac{t^2}{8 (j+1) n^{1-2\delta}}\right).\]
	In particular, taking $t = n^{1/2 -\eps_0}$ we obtain
	\begin{align} 
	\label{eq:ah}
	\Pr \left( \max_{0\leq j \leq n^{\beta}}|T'_{n,j}|\geq n^{1/2-\eps_0}\right) & 
	\leq (n^\beta +1)\max_{0\leq j \leq n^\beta} \Pr\left(|T'_{n,j}|\geq n^{1/2-\eps_0} \right) \nonumber\\
	&\leq 2 (n^\beta +1) \exp\left( - \frac{n^{1-2\eps_0}}{16 n^{1+\beta-2\delta}}\right),  
	\end{align}
	for all $n$ sufficiently large. Now choose and fix $\delta = \delta (p) := \min \{ 1/4,  (p-2)/4p \}$,
	so $\delta >0$ satisfies the bounds earlier in this proof, and then choose
	$\beta < \beta_0 := \delta$ such that 
	\[ \frac{n^{1-2\eps_0}}{n^{1+\beta-2\delta}} = n^{2\delta - 2 \eps_0 - \beta} \geq n^{\delta - 2\eps_0} .\]
	So choosing $\eps_0 = \delta /4$ we have that 
	the probability bound in~\eqref{eq:ah} is summable. Thus by the Borel--Cantelli lemma, we have that
	$\max_{0\leq j \leq n^\beta} |T'_{n,j}| \leq n^{1/2-\eps_0}$ for all but finitely many $n$, a.s.
		It follows that, for any $\eps \in (0,\eps_0)$,
			\[\lim_{n\rightarrow \infty} \frac{|\sum_{i=n-n^\beta}^{n}(\xi'_i-\Exp\xi'_i)|}{n^{1/2-\eps}}=0, \as,\]
	which completes the proof.
\end{proof}

\section*{Acknowledgements}

The authors are grateful to Ostap Hryniv,
for numerous helpful discussions on the subject of this paper, and to
Wilfrid Kendall for a question that prompted us to formulate Theorem~\ref{thm:shape}.
The first author is supported by an EPSRC
studentship (EP/M507854/1).

\end{document}